\documentclass[final,leqno]{amsart}
\usepackage{amsmath,amssymb,amsfonts,latexsym,stmaryrd}
\usepackage{graphicx,float,epsfig} 
\usepackage{mathrsfs} 
\usepackage{float} 
\usepackage{color}
\usepackage{setspace} 

\DeclareGraphicsExtensions{ .eps,.ps} \usepackage{subfig}
\usepackage{multirow}
\usepackage{url}
\usepackage[linesnumbered,algoruled,boxed]{algorithm2e}
\usepackage[colorlinks=false, linkcolor=blue,  citecolor=OliveGreen,  pdfstartview={}{}]{hyperref} 
\usepackage[dvipsnames]{xcolor}

\usepackage{svg}
\usepackage{amsmath}

\def\O{\Omega}

\def\CM{\mathcal{X}}
\def\CN{\mathcal{Y}}

\renewcommand\sp{\mathop{\mathrm{Sp}}\nolimits}
\newtheorem{remark}{Remark}[section]
\newtheorem{lemma}{Lemma}[section]
\newtheorem{theorem}{Theorem}[section]
\newtheorem{prop}{Proposition}[section]
\newtheorem{assumption}{Assumption}[section]
\usepackage{booktabs}
\usepackage{float}

\newcommand\bu{\boldsymbol{u}}
\newcommand\bv{\boldsymbol{v}}

\def\hdel{\widehat{\delta}}
\def\CM{\mathcal{X}}
\def\CN{\mathcal{Y}}

\newcommand\bI{\boldsymbol{I}}

\newcommand\bS{\boldsymbol{S}}
\newcommand\bT{\boldsymbol{T}}

\newcommand\bQ{\boldsymbol{Q}}

\newcommand\bxi{\boldsymbol{\xi}}

\def\CT{{\mathcal T}}

\newcommand{\dd}{\texttt{d}}

\newcommand\bsig{\boldsymbol{\sigma}}
\newcommand\btau{\boldsymbol{\tau}}
\newcommand\bPi{\boldsymbol{\Pi}}

\newcommand\R{\mathbb{R}}

\renewcommand\H{\mathrm{H}}
\renewcommand\L{\mathrm{L}}

\renewcommand\O{\Omega}

\newcommand\bdiv{\mathop{\mathbf{div}}\nolimits}

\renewcommand\div{\mathop{\mathrm{div}}\nolimits}

\newcommand\tD{\mathtt{D}}
\newcommand\tr{\mathop{\mathrm{tr}}\nolimits}

\renewcommand\sp{\mathop{\mathrm{sp}}\nolimits}
\newcommand\dist{\mathop{\mathrm{dist}}\nolimits}

\newcommand\LO{\L^2(\O)}

\newcommand\ws{\widehat{s}}

\newcommand\ds{\displaystyle}

\newcommand{\vertiii}[1]{{\left\vert\kern-0.25ex\left\vert\kern-0.25ex\left\vert #1 
    \right\vert\kern-0.25ex\right\vert\kern-0.25ex\right\vert}}

\begin{document}

\title[Mixed formulation for the elasticity eigenproblem]
{Displacement-pseudostress formulation for the linear elasticity spectral problem}

\author{Daniel Inzunza}
\address{CI$^2$MA and Departamento de Ingenier\'ia Matem\'atica, Facultad de Ciencias F\'isicas y Matem\'aticas,
Universidad de Concepci\'on, Casilla 160-C, Concepci\'on, Chile.}
\email{dinzunza@ing-mat.udec.cl }
\thanks{The first author was partially supported by CONICYT/PAI/Concurso Apoyo a Centros Cient\'ificos y Tecnol\'ogicos de Excelencia con Financiamiento Basal AFB 170001.}
\author{Felipe Lepe}
\address{GIMNAP-Departamento de Matem\'atica, Universidad del B\'io - B\'io, Casilla 5-C, Concepci\'on, Chile.}
\email{flepe@ubiobio.cl}
\thanks{The second author was partially supported by
ANID-Chile through FONDECYT project 11200529 (Chile).}

\author{Gonzalo Rivera}
\address{Departamento de Ciencias Exactas,
Universidad de Los Lagos, Casilla 933, Osorno, Chile.}
\email{gonzalo.rivera@ulagos.cl}
\thanks{The third author was supported by
ANID-Chile through FONDECYT project 11170534 (Chile).}




\begin{abstract}
In this paper we analyze a mixed  displacement-pseudostress formulation
for the elasticity eigenvalue problem. We propose a finite element method
to approximate the pseudostress tensor with Raviart-Thomas elements
and the displacement with piecewise polynomials.
 With the aid of the classic  theory for compact operators, we prove that our method is convergent and does not introduce spurious modes. Also, we obtain error estimates for the proposed method. Finally, we report some numerical tests supporting the theoretical results. 
\end{abstract}

\maketitle

\section{Introduction}\label{sec:intro}
The linear elasticity equations are an important subject of study for engineers and mathematicians that describes the displacement of some structure with elastic properties. For a given domain $\O\subset\mathbb{R}^n$, $n\in\{2,3\}$, with Lipschitz boundary $\partial\O$, we are interested in the elasticity eigenvalue problem: Find $\kappa\in\mathbb{R}$ and the pair $(\bsig,\bu)$ such that
\begin{equation}\label{def:elast_system}
\left\{
\begin{array}{rcll}
\bsig & = & 2\mu\boldsymbol{\varepsilon}(\bu)+\lambda\tr(\boldsymbol{\varepsilon}(\bu))\mathbb{I}&  \text{ in } \quad \Omega, \\
\div\bsig & = & -\kappa\bu & \text{ in } \quad \Omega, \\
\bu & = & \mathbf{0} & \text{ on } \quad \partial\Omega,
\end{array}
\right.
\end{equation}
where $\bu$ represents the displacement of the elastic structure, $\bsig$ is the Cauchy symmetric tensor, $\lambda$ and $\mu$ are  the positive Lam\'e constants, $\mathbb{I}\in\mathbb{R}^{n\times n}$ is the identity matrix and $\boldsymbol{\varepsilon}(\bu)$ represents the tensor of small deformations,  given by $\boldsymbol{\varepsilon}(\bu):=\frac{1}{2}(\nabla\bu+(\nabla\bu)^{\texttt{t}})$, where $\texttt{t}$ is the transpose operator. It is well known that the Lam\'e constant $\lambda$ depends on the Poisson's ratio of the structure which, when it tends  to $1/2$, produces that $\lambda$ tends to infinity, introducing instabilities in the numerical methods like the locking effect. 

The importance of approximate  the eigenmodes of  system \eqref{def:elast_system} lies in the fact that the stability of different elastic structures used in real applications, like beams, rods, plates, just for mention a few, depend on the accurate knowledge of the vibration modes of these structures.

With the aim of approximate the solutions of the linear elasticity equations, several numerical methods have been designed, firstly for the load problem in the past years. We refer to \cite{MR3715319, MR3860570, MR2824163, MR3453481,MR2831058,MR3962898,MR3036997,MR4050542}, and the reference therein, just to mention some results on these subjects.

In this sense, there are different formulations to study the spectral linear elasticity problem, where different unknowns are introduced in order to obtain the most complete information about the response of the elastic structures. It is clear that the main unknowns are the displacement and the Cauchy
stress tensor. However,   new formulations have been analyzed where additional unknowns are introduced. For example, in \cite{MR3036997} the authors introduce a mixed formulation depending only on the Cauchy stress tensor, where the symmetry is weakly  imposed  and the displacement can be recovered with a post-process. Analysis with a discontinuous Galerkin method (DG) for this formulation has been also proposed in  \cite{MR3962898} for the elasticity spectral problem, where the advantages of considering more general meshes are presented. Nevertheless, the main disadvantage of this method lies in the correct choice of the stabilization parameter, since, depending on the configuration of the problem, namely the geometry, boundary conditions or physical quantities, it can generate  spurious eigenvalues in the computed spectrum. This has been also observed in other problems where the methods need to be stabilized for some  parameter, as it occurs in  \cite{MR3660301,MR4077220,MR4253143, MR3340705}, just for mention some recent papers that deal with this subject.

The additional costs that these new methods bring due their nature,  are not a difficulty when spectral problems are solved with the classic finite element method (FEM), since there are not a dependency
on some other parameters when we are approximating the spectrum of the solution operators. This is a clear advantage of the FEM, and is  the motivation of the present paper.  Specifically, the purpose of this work is to demonstrate the advantages of applying standard mixed finite elements to tensorial formulations in eigenvalue problems. More precisely, the formulation introduced in \cite{ MR3453481} for problem \eqref{def:elast_system}, where the main unknowns are the  displacement of the structure and its pseudostress. These pseudostress formulations, previously introduced in \cite{MR2642330,MR2594823, MR2878511} in contexts unrelated to eigenvalue problems, have been  subject of attention in the community since this tensor 
allows to approximate other variables as the gradients of the velocity and  pressure in flow problems, and the Cauchy stress tensor or the strain tensor for linear elasticity, among others.

The proposed mixed element method  approximates the  pseudostress tensor  with Raviart-Thomas elements, which must be understood in the tensor context and the displacement with piecewise polynomials, both of order $k\geq 0$. With these mixed method,  we approximate the spectrum and the corresponding eigenfunctions, but also we conclude that the method does introduce spurious eigenvalues and delivers an accurate approximation of the spectrum.   In addition, unlike \cite{MR3962898,MR3036997} where the authors have used the theory for non-compact operators to analyze the elasticity spectral problem, since the solution operators in these references are defined in $\mathbb{H}(\bdiv)$,  we will use the classic theory for compact operators due the simplicity of the proposed solution operator that is defined only in $\L^2(\Omega)$.  

The paper is organized as follows: In Section \ref{sec:model_problem} we present the elasticity eigenvalue problem and its pseudostress-displacement formulation. We recall some important properties. Also we introduce the corresponding solution operator and its corresponding spectral characterization. In Section \ref{sec:mixed_method} we present  the mixed element method for our spectral problem. We recall some approximation properties, ad-hoc for the regularity results established in the previous section, and analyze the stability of the mixed method for the eigenvalue problem. We introduce the discrete solution operator. In Section
 \ref{sec:error} we analyze the convergence of our method, by applying the results of \cite{MR1115235}. Also, we prove error estimates for the eigenvalues and eigenfunctions. Finally, in Section \ref{sec:numerics}, we report a set of numerical experiments that allow us to assess the convergence properties of the method.

We end this section with some notations that will be used below. Given $n\in\{ 2,3\}$, we denote $\mathbb{R}^{n\times n}$ the space of vectors and tensors of order $n$ with entries in $\mathbb{R}$, and $\mathbb{I}$ is the identity matrix of $\mathbb{R}^{n\times n}$. Given any $\boldsymbol{\tau}:=(\tau_{ij})$ and $\boldsymbol{\sigma}:=(\sigma_{ij})\in \mathbb{R}^{n\times n}$, we write
$$
	\boldsymbol{\tau}^{\texttt{t}}:=(\tau_{ji}), \quad \tr(\boldsymbol{\tau}):=\sum_{i=1}^{n}\tau_{ii}, \quad \mbox{and} \quad \boldsymbol{\tau:\sigma}:=\sum_{i,j=1}^{n} \tau_{ij}\,\sigma_{ij}, 
$$
to refer to the transpose, the trace and the tensorial product between $\boldsymbol{\tau}$ and $\boldsymbol{\sigma}$ respectively. 

For $s\geq 0$, we denote as $\| \cdot \|_{s, \O}$ the norm of the Sobolev space $\H^{s}(\O)$ or
$\mathbb{H}^s(\O):=\H^{s}(\O)^{n\times n}$ with n=2,3,  for scalar  and tensorial fields, respectively, with the convention $\H^0(\O):=\LO$ and $\mathbb{H}^0(\O):=\mathbb{L}^2(\O)$. Furthermore, with $\div$ denoting the usual divergence operator, we define the Hilbert space
\[
	\H(\bdiv, \O):=\{ \boldsymbol{\tau} \in \L^{2}(\O) \,:\, \div(\boldsymbol{\tau}) \in \L^{2}(\O) \},
\] 
endowed with the norm $ \| \boldsymbol{\tau} \|_{\bdiv, \O}^{2}:= \|\boldsymbol{\tau} \|_{0,\O}^{2} + \| \div(\boldsymbol{\tau}) \|_{0,\O}^{2} $. The space of matrix valued functions whose rows belong to $\H(\div, \O)$ will be denoted by $\mathbb{H}(\bdiv, \O)$ where $\bdiv$ stands for the action of the divergence operator along on each row of a tensor.

Finally, we use $C$ with or without subscripts, bar, tildes or hats, to denote  generic constants independent of the discretization parameter, which may take different values at different places.  

%
\section{The model problem}
\label{sec:model_problem}

This section is dedicated to describe the model problem in which our method will be based. From the first equation of
\eqref{def:elast_system} we have that
\begin{equation*}
\bdiv\bsig=2\mu\bdiv\boldsymbol{\varepsilon}(\bu)+\lambda\nabla\div\bu=\mu\Delta\bu+(\lambda+\mu)\nabla\div\bu.
\end{equation*}
This allows to rewrite \eqref{def:elast_system} as follows
\begin{equation*}\label{def:elast_system_reduced}
\left\{
\begin{array}{rcll}
\mu\Delta\bu+(\lambda+\mu)\nabla\div\bu & = & -\kappa\bu &  \text{ in } \quad \Omega, \\
\bu & = & \mathbf{0} & \text{ on } \quad \partial\Omega.
\end{array}
\right.
\end{equation*}

Now we introduce the so called pseudostress tensor, defined by
\begin{equation*}
\label{eq:pseudo}
\boldsymbol{\rho}:=\mu\nabla\bu+(\lambda+\mu)\div\bu\mathbb{I}=\mu\nabla\bu+(\lambda+\mu)\tr(\nabla\bu)\mathbb{I}.
\end{equation*}

Observe that $\bdiv\bsig=\bdiv\boldsymbol{\rho}$. Hence, we have the following formulation where the pseudostress and the displacement are the main unknowns: Find $\kappa\in\mathbb{R}$ and $(\boldsymbol{\rho}, \bu)$ such that

\begin{equation}\label{def:elast_system_rho_1}
\left\{
\begin{array}{rcll}
\boldsymbol{\rho} & = &\mu\nabla\bu+(\lambda+\mu)\tr(\nabla\bu)\mathbb{I}&  \text{ in } \quad \Omega, \\
\bdiv\boldsymbol{\rho} & = & -\kappa\bu & \text{ in } \quad \Omega, \\
\bu & = & \mathbf{0} & \text{ on } \quad \partial\Omega.
\end{array}
\right.
\end{equation}

Moreover, the following identity holds (see \cite[Section 2]{MR3453481} for details)
\begin{equation*}
\displaystyle\frac{1}{\mu}\left\{\boldsymbol{\rho}-\frac{\lambda+\mu}{n\lambda+(n+1)\mu}\tr(\boldsymbol{\rho})\mathbb{I} \right\}=\nabla\bu.
\end{equation*}

This leads to the following eigenvalue problem
\begin{equation}\label{def:elast_system_rho}
\left\{
\begin{array}{rcll}
\displaystyle\frac{1}{\mu}\left\{\boldsymbol{\rho}-\frac{\lambda+\mu}{n\lambda+(n+1)\mu}\tr(\boldsymbol{\rho})\mathbb{I} \right\}& = &\nabla\bu&  \text{ in } \quad \Omega, \\
\bdiv\boldsymbol{\rho} & = & -\kappa\bu & \text{ in } \quad \Omega, \\
\bu & = & \mathbf{0} & \text{ on } \quad \partial\Omega.
\end{array}
\right.
\end{equation}

Multiplying the above system with suitable tests functions, integrating by parts and using the boundary condition, we obtain the following variational formulation: Find $\kappa\in\mathbb{R}$ and $\boldsymbol{0}\neq (\boldsymbol{\rho},\bu)\in\mathbb{H}\times \mathbf{Q}$, such that

\begin{equation}\label{def:spectral_1}
\left\{
\begin{array}{rcll}
a(\boldsymbol{\rho},\btau)+b(\btau,\bu) & = &0&  \forall\btau\in\mathbb{H}, \\
b(\boldsymbol{\rho},\bv)& = & -\kappa(\bu,\bv)_{0,\O} &  \forall\bv\in \mathbf{Q},
\end{array}
\right.
\end{equation}
where $\mathbb{H}:=\mathbb{H}(\bdiv;\O)$ and $\mathbf{Q}:=\L^2(\O)^n$ and the bilinear forms $a:\mathbb{H}\times\mathbb{H}\rightarrow \mathbb{R}$ and $b:\mathbb{H}\times \mathbf{Q}\rightarrow\mathbb{R}$ are defined by
\begin{equation*}
\displaystyle a(\bxi,\btau):=\frac{1}{\mu}\int_{\Omega}\bxi:\btau-\frac{\lambda+\mu}{\mu(n\lambda+(n+1)\mu)}\int_{\O}\tr(\bxi)\tr(\btau)\quad\forall\bxi,\btau\in\mathbb{H},
\end{equation*}
and 
\begin{equation*}
b(\btau,\bv):=\int_{\O}\bv\cdot\bdiv\btau\quad\forall\btau\in\mathbb{H},\,\,\forall\bv\in \mathbf{Q}.
\end{equation*}

For $\btau\in\mathbb{H}$ we define its associated deviator tensor by $\btau^{\texttt{d}}:=\btau-\frac{1}{n}\tr(\btau)\mathbb{I}$, which allows us to redefine $a(\cdot,\cdot)$ as follows
\begin{equation}
\label{eq:identity_a}
\displaystyle a(\bxi,\btau):=\frac{1}{\mu}\int_{\Omega}\bxi^{\texttt{d}}:\btau^{\texttt{d}}+\frac{1}{n(n\lambda+(n+1)\mu)}\int_{\O}\tr(\bxi)\tr(\btau)\quad\forall\bxi,\btau\in\mathbb{H}.
\end{equation}

With the purpose of establish the stability of the mixed formulation \eqref{def:spectral_1}, we introduce the following decomposition $\mathbb{H}:=\mathbb{H}_0\oplus \R \mathbb{I}$ where
\begin{equation*}
\mathbb{H}_0:=\left\{\btau\in\mathbb{H}\,:\,\int_{\O}\tr(\btau)=0\right\}.
\end{equation*}
Note that for any $\bxi\in\mathbb{H}$ there exist a unique $\bxi_{0}\in \mathbb{H}_0$ and $d:=\dfrac{1}{n|\O|}\displaystyle\int_{\O}\tr(\bxi)\in\R$ such that $\bxi=\bxi_{0}+d \, \mathbb{I}$.

Observe that in our case due the vanishing Dirichlet condition on \eqref{def:elast_system_rho_1} we have that
$d=0$ (see \cite[Lemma 2.1]{MR3453481}). Hence, we are in position to work indistinctly with $\mathbb{H}$ or $\mathbb{H}_0$.

We invoke the  following result (see \cite[Ch. 4, Proposition 3.1]{MR3097958})
\begin{equation}
\label{eq:des_deviator}
\|\btau\|_{0,\O}^2\leq C \|\btau^{\dd}\|_{0,\O}^2+\|\bdiv\btau\|_{0,\O}^2\quad\forall\btau\in\mathbb{H}_0.
\end{equation}

It is easy to check that $a(\cdot,\cdot)$ and $b(\cdot,\cdot)$ are bounded bilinear forms (see \cite[Theorem 2.1]{MR3453481}). On the other hand, let $\mathbb{V}$ be the kernel of $b(\cdot,\cdot)$ (namely, the inducted operator by this bilinear form), defined by $\mathbb{V}:=\{\btau\in\mathbb{H}_0\,:\,\,\bdiv(\btau)=\boldsymbol{0}\}$. With this space at hand, it is easy to check that there exists $\alpha>0$ such that the following coercivity result holds
\begin{equation*}
a(\btau,\btau)\geq \alpha\|\btau\|_{\bdiv,\O}^2\quad\forall\btau\in\mathbb{V}.
\end{equation*}

On the other hand, the  following inf-sup condition for $b(\cdot,\cdot)$ holds  (see \cite[Theorem 2.1]{MR3453481}), 
\begin{equation}
\label{eq:in_sup_cont}
\displaystyle\sup_{\boldsymbol{0}\neq\btau\in \mathbb{H}_{0}}\frac{b(\btau,\bv)}{\|\btau\|_{\bdiv,\O}}\geq\beta\|\bv\|_{0,\O}\quad\forall\bv\in \mathbf{Q}.
\end{equation}

Hence, according to the Babu\^{s}ka-Brezzi theory (see \cite{MR3097958} for a complete revision about this theory), problem \eqref{def:spectral_1} is well defined. 

All the previous results are sufficient to introduce the so called solution operator that relates the spectral problem   \eqref{def:spectral_1} with its associated source problem. We consider in our work the following operator
\begin{align*}
\bT_{\lambda}:\mathbf{Q}&\rightarrow\mathbf{Q},\\
           \boldsymbol{f}&\mapsto \bT_{\lambda}\boldsymbol{f}:=\widehat{\bu}, 
\end{align*}
where the pair $(\widehat{\boldsymbol{\rho}}, \widehat{\bu})$ is the solution of the following source problem
 \begin{equation}\label{def:sourcel_1}
\left\{
\begin{array}{rcll}
a(\widehat{\boldsymbol{\rho}},\btau)+b(\btau,\widehat{\bu}) & = &0&  \forall\btau\in\mathbb{H}_{0}, \\
b(\widehat{\boldsymbol{\rho}},\bv)& = & - (\boldsymbol{f},\bv)_{0,\O} &  \forall\bv\in \mathbf{Q}.
\end{array}
\right.
\end{equation}

As a consequence of the Babu\^{s}ka-Brezzi theory, we have that $\bT_{\lambda}$ is well defined.
Moreover, it is easy to check that $\bT_{\lambda}$ is self-adjoint with respect to  the $\L^2(\O)$ inner product. Indeed, given
$\boldsymbol{f},\widehat{\boldsymbol{f}}\in\mathbf{Q} $, let $(\widehat{\boldsymbol{\rho}},\widehat{\bu})\in\mathbb{H}_{0}\times\mathbf{Q}$ and $(\widetilde{\boldsymbol{\rho}},\widetilde{\bu})\in\mathbb{H}_{0}\times\mathbf{Q}$  be the solutions to problem \eqref{def:sourcel_1} with right hand sides $\boldsymbol{f}$ and $\widehat{\boldsymbol{f}}$, respectively. Assume that that $\bT_{\lambda}\boldsymbol{f}=\widehat{\bu}$ and $\bT_{\lambda}\widehat{\boldsymbol{f}}=\widetilde{\bu}$.  The symmetry of $a(\cdot,\cdot)$ and $(\cdot,\cdot)_{0,\O}$ implies that
\begin{multline*}
(\boldsymbol{f},\bT_{\lambda}\widehat{\boldsymbol{f}})_{0,\O}=(\boldsymbol{f},\widetilde{\bu})_{0,\O}\\=
-\big(a(\widehat{\boldsymbol{\rho}},\widetilde{\bu})+b(\widetilde{\bu}, \widehat{\boldsymbol{\rho}})+b(\widehat{\boldsymbol{\rho}},\widetilde{\bu}) \big)=
(\widehat{\boldsymbol{f}}, \widehat{\bu})_{0,\O}=(\bT_{\lambda}\boldsymbol{f}, \widehat{\boldsymbol{f}})_{0,\O}.
\end{multline*}

We invoke the following estimate  (\cite[Theorem 2.1 ]{MR3453481}):  there exists constant $C>0$,  independent of $h$ and $\lambda$ such that 
\begin{equation}
\label{eq_cotasupfuente}
\|\widehat{\boldsymbol{\rho}}\|_{\bdiv,\O}+\|\widehat{\bu}\|_{0,\O}\leq C \|\boldsymbol{f}\|_{0,\O}.
\end{equation}

It is direct that $(\kappa, (\boldsymbol{\rho},\bu))\in\mathbb{R}\times \mathbb{H}_{0}\times \mathbf{Q}$ solves \eqref{def:spectral_1} if and only if $(\zeta,\bu)$ is an eigenpair of $\bT_{\lambda}$, i.e.,
\begin{equation*}
\ds \bT_{\lambda}\bu=\zeta\bu\quad\text{with}\,\,\zeta:=\frac{1}{\kappa} \quad\text{and}\,\,\zeta\neq 0.
\end{equation*}

Now we present an additional regularity for the eigenfunctions of $\bT_{\lambda}$ which is derived from the classic regularity results for linear elasticity (see \cite{MR840970}), together with a standard bootstrap argument.  
\begin{lemma}[Regularity of the eigenfunctions]
\label{lmm:add_eigen}
Let $\bu$ be an eigenfunction of $\bT_{\lambda}$ associated to an eigenvalue $\kappa$. Then, 
for all $s\in (0,\ws)$, where $\ws>0$, we have that $\bu\in\H^{1+s}(\Omega)^n$. Also, there exists a constant $C>0$ which in principle depends on $\lambda$, such that
\begin{equation*}
\|\bu\|_{1+s,\O}\leq \widehat{C}\|\bu\|_{0,\O}.
\end{equation*}
\end{lemma}

\begin{remark} \label{daniel0}
Observe that lemma above, in conjunction with the first equation of \eqref{def:elast_system_rho_1}, implies immediately that $\boldsymbol{\rho}\in  \mathbb{H}^{s}(\O)$. On the other hand, for the divergence term, it is enough to consider the second equation in \eqref{def:elast_system_rho} to deduce that $\bdiv\boldsymbol{\rho}\in \H^{1+s}(\O)^{n}$.
\end{remark}

We mention that the dependency of the constants in the regularity exponents and boundedness on $\lambda$ is not completely evident, since in our numerical experiments (cf. Section \ref{sec:numerics}), even in the limit case ($\lambda=\infty$), our method obtains the expected convergence orders. This leads us to consider  the following assumption along our paper:
\begin{assumption}
Constants $\widehat{s}$ and $\widehat{C}$ in Lemma \ref{lmm:add_eigen} are independent of $\lambda$.
\end{assumption}

Finally,  the spectral characterization of $\bT_{\lambda}$ is the following.
\begin{theorem}[Spectral characterization of $\bT_{\lambda}$]
\label{thrm:spec_char_T}
The spectrum of $\bT_{\lambda}$ satisfies $\sp(\bT_{\lambda})=\{0\}\cup\{\zeta_k\}_{k\in\mathbb{N}}$, where $\{\zeta_k\}_{k\in\mathbb{N}}$
is a sequence of real positive eigenvalues which converges to zero, repeated according their respective multiplicities. 
\end{theorem}

It is important to take into account the fact that the coefficient $\lambda$ in the elasticity eigenproblem leads to the analysis of a family of
problems where for every choice of $\lambda$, we solve a different eigenvalue problem. 

A natural question is what happens with the spectrum of problem \eqref{def:spectral_1} when $\lambda$ goes to infinity. To answer this, we will analyze the limit eigenvalue problem.
\subsection{The limit problem} The elasticity eigenvalue problem has the particularity that 
when $\nu\rightarrow 1/2$, the Lam\'e constant $\lambda\rightarrow+\infty$. This is an interesting 
case, since when $\lambda=+\infty$, the nearly incompressible elasticity eigenvalue problem becomes the perfectly  incompressible elasticity eigenvalue problem and hence, the respective spectrums will converge to each other.

Let us introduce the limit problem: Find $\kappa_{\infty}\in\mathbb{R}$ and $(\boldsymbol{\rho}_{\infty}, \bu_{\infty})\in\mathbb{H}_0\times\mathbf{Q}$ such that
\begin{equation}\label{def:spectral_limit}
\left\{
\begin{array}{rcll}
a(\boldsymbol{\rho}_{\infty},\btau)+b(\btau,\bu_{\infty}) & = &0&  \forall\btau\in\mathbb{H}_{0}, \\
b(\boldsymbol{\rho}_{\infty},\bv)& = & -\kappa_{\infty}(\bu_{\infty},\bv)_{0,\O} &  \forall\bv\in \mathbf{Q}.
\end{array}
\right.
\end{equation}

Let us remark that since $\lambda=\infty$, the bilinear form $a(\cdot,\cdot)$ in \eqref{def:spectral_limit}
consists only in the term $\int_{\O}\boldsymbol{\rho}_{\infty}^{\tD}:\boldsymbol{\tau}^{\tD}$, whereas $b(\cdot,\cdot)$ have no changes on its definition.

Now we are in position to introduce the solution operator associated to \eqref{def:spectral_limit}
\begin{align*}
\bT_{\infty}:\mathbf{Q}&\rightarrow\mathbf{Q},\\
           \boldsymbol{f}&\mapsto \bT_{\infty}\boldsymbol{f}:=\widehat{\bu}_{\infty}, 
\end{align*}
where $(\widehat{\boldsymbol{\rho}}_{\infty},\widehat{\bu}_{\infty})$ is the solution of the following source problem
\begin{equation}\label{def:source_limit}
\left\{
\begin{array}{rcll}
a(\widehat{\boldsymbol{\rho}}_{\infty},\btau)+b(\btau,\widehat{\bu}_{\infty}) & = &0&  \forall\btau\in\mathbb{H}_{0}, \\
b(\widehat{\boldsymbol{\rho}}_{\infty},\bv)& = & (\boldsymbol{f},\bv)_{0,\O} &  \forall\bv\in \mathbf{Q}.
\end{array}
\right.
\end{equation}

Similar to the regularity properties demonstrated for the operator $\bT_{\lambda}$, the following results are reported:
 Now, using the relation between incompressible elasticity and the Stokes problem and according to  \cite{MR851383}  we conclude that: there exists $\widehat{s}_{\infty}\in (0,1)$ and a constant $\widehat{C}>0$ depending on the domain and $\mu$, such that $\widehat{\bu}_{\infty}\in\H^{1+s}(\Omega)^n$ and
\begin{equation*}
\label{eq:reg_u_infty}
||\widehat{\boldsymbol{\rho}}_{\infty}\|_{s,\O}+\|\widehat{\bu}_{\infty}\|_{1+s,\Omega}\leq\widehat{C}\|\boldsymbol{f}\|_{0,\Omega}\quad\forall s\in (0,\widehat{s}_{\infty}).
\end{equation*}
Also,  the operator $\bT_{\infty}$ is self-adjoint, well defined and compact implying that its   spectrum consists in a sequence of real eigenvalues $\{\kappa_{\infty_k}\}_{k\in\mathbb{N}}$ that converge to zero.
 

The main result of this section is the following.
\begin{lemma}[convergence of $\bT_{\lambda}$ to $\bT_{\infty}$]
There exists a constant $C>0$ such that
\begin{equation*}
\|(\bT_{\lambda}-\bT_{\infty})\boldsymbol{f}\|_{0,\O}\leq \frac{C}{\lambda}\|\boldsymbol{f}\|_{0,\O}\quad\forall\boldsymbol{f}\in\mathbf{Q}.
\end{equation*}
\end{lemma}
\begin{proof}
Let $\boldsymbol{f}\in\mathbf{Q}$ and let $\bu:=\bT_{\lambda}\boldsymbol{f}$ and $\bu_{\infty}:=\bT_{\infty}\boldsymbol{f}$.
Subtracting problems \eqref{def:sourcel_1} and \eqref{def:source_limit} we have
\begin{align*}
\frac{1}{\mu}\int_{\O}(\boldsymbol{\rho}^{\tD}-\boldsymbol{\rho}_{\infty}^{\tD}):\btau+\frac{1}{n(n\lambda+(n+1)\mu)}\int_{\O}\tr(\boldsymbol{\rho})\tr(\btau)+\int_{\O}\bdiv\btau\cdot(\bu-\bu_{\infty})&=0,\\
\int_{\O}\bdiv(\boldsymbol{\rho}-\boldsymbol{\rho}_{\infty})\cdot\bv&=0,
\end{align*}
for all $\btau\in\mathbb{H}_0$ and for all $\bv\in\mathbf{Q}$.

Set $\btau=\boldsymbol{\rho}-\boldsymbol{\rho}_{\infty}$ and $\bv=\bu-\bu_{\infty}$ in problem above. Then we have
\begin{multline}
\label{eq:limit_rho}
\displaystyle\frac{1}{\mu}\|\boldsymbol{\rho}^{\tD}-\boldsymbol{\rho}^{\tD}_{\infty}\|^2_{0,\O}=-\frac{1}{n(n\lambda+(n+1)\mu)}\int_{\O}\tr(\boldsymbol{\rho})\tr(\boldsymbol{\rho}-\boldsymbol{\rho}_{\infty})\\
\leq \frac{1}{n\lambda}\|\boldsymbol{\rho}\|_{0,\O}\|\boldsymbol{\rho}-\boldsymbol{\rho}_{\infty}\|_{0,\O},
\end{multline}
where we have used the inequality $\|\tr(\btau)\|_{0,\O}\leq \sqrt{n}\|\btau\|_{0,\O}$. Since $\boldsymbol{\rho}$ solves \eqref{def:sourcel_1}, we have that there exists a constant $C>0$ such that $\|\boldsymbol{\rho}\|_{0,\O}\leq C\|\boldsymbol{f}\|_{0,\O}$. Replacing this in \eqref{eq:limit_rho} we have
\begin{equation}
\label{eq:bound_deviators}
\displaystyle\frac{1}{\mu}\|\boldsymbol{\rho}^{\tD}-\boldsymbol{\rho}^{\tD}_{\infty}\|^2_{0,\O}\leq \frac{C}{n\lambda}\|\boldsymbol{f}\|_{0,\O}\|\boldsymbol{\rho}-\boldsymbol{\rho}_{\infty}\|_{0,\O}.
\end{equation}
From \eqref{eq:des_deviator} we have
\begin{equation*}
\|\boldsymbol{\rho}-\boldsymbol{\rho}_{\infty}\|^2\leq C\|\boldsymbol{\rho}^{\tD}-\boldsymbol{\rho}_{\infty}^{\tD}\|_{0,\O}^2+\|\bdiv(\boldsymbol{\rho}-\boldsymbol{\rho}_{\infty})\|_{0,\O}^2,
\end{equation*}
which implies immediately  that $\|\boldsymbol{\rho}-\boldsymbol{\rho}_{\infty}\|_{0,\O}\leq\|\boldsymbol{\rho}^{\tD}-\boldsymbol{\rho}_{\infty}^{\tD}\|_{0,\O}$. Replacing this in \eqref{eq:bound_deviators} leads to 
\begin{equation}
\label{eq:cota_desv_f}
\displaystyle\frac{1}{\mu}\|\boldsymbol{\rho}^{\tD}-\boldsymbol{\rho}^{\tD}_{\infty}\|_{0,\O}\leq \frac{C}{n\lambda}\|\boldsymbol{f}\|_{0,\O}.
\end{equation}

On the other hand, from the inf-sup condition \eqref{eq:in_sup_cont}, Cauchy-Schwarz inequality and \eqref{eq:cota_desv_f} we obtain
\begin{align*}
&\beta\|\bu-\bu_{\infty}\|_{0,\O}\leq\displaystyle\sup_{\boldsymbol{0}\neq\btau\in \mathbb{H}_{0}}\frac{b(\btau,\bu-\bu_{\infty})}{\|\btau\|_{\bdiv,\O}}\\
&=\displaystyle\sup_{\boldsymbol{0}\neq\btau\in \mathbb{H}_{0}}\frac{-a(\boldsymbol{\rho}-\boldsymbol{\rho}_{\infty}, \btau)}{\|\btau\|_{\bdiv,\O}}\\
&=\displaystyle\sup_{\boldsymbol{0}\neq\btau\in \mathbb{H}_{0}}\frac{\displaystyle-\frac{1}{\mu}\int_{\O}(\boldsymbol{\rho}-\boldsymbol{\rho}_{\infty})^{\tD}:\btau^{\tD}-\frac{1}{n(n\lambda+(n+1)\mu)}\int_{\O}\tr(\boldsymbol{\rho}-\boldsymbol{\rho}_{\infty})\tr(\btau)}{\|\btau\|_{\bdiv,\O}}\\
&\leq \displaystyle\sup_{\boldsymbol{0}\neq\btau\in \mathbb{H}_{0}}\frac{\displaystyle\frac{1}{\mu}\|(\boldsymbol{\rho}-\boldsymbol{\rho}_{\infty})^{\tD}\|_{0,\O}\|\btau^{\tD}\|_{0,\O}+\frac{1}{n(n\mu+\mu)}\|\tr(\boldsymbol{\rho}-\boldsymbol{\rho}_{\infty})\|_{0,\O}\|\tr(\btau)\|_{0,\O}}{\|\btau\|_{\bdiv,\O}}\\
&\leq \displaystyle\sup_{\boldsymbol{0}\neq\btau\in \mathbb{H}_{0}}\frac{\displaystyle\frac{C\left(n+\sqrt{n}\right)}{n^2\lambda}\|\boldsymbol{f}\|_{0,\O}\|\btau\|_{0,\O}+\frac{1}{n(n+1)}\frac{C}{n\lambda}\|\boldsymbol{f}\|_{0,\O}\|\btau\|_{0,\O}}{\|\btau\|_{\bdiv,\O}}\\
&\leq \frac{C_*}{\lambda}\|\boldsymbol{f}\|_{0,\O},
\end{align*}
where $\displaystyle C_*=\frac{C\left(1+n+\sqrt{n}\right)}{n^2}$.  This concludes the proof.
\end{proof}

We end this section presenting a well known consequence of the convergence in norm established in the previous lemma (see \cite{ MR1115235}, for instance).
\begin{theorem}
Let $\xi_{\infty}>0$ be an eigenvalue of $\bT_{\infty}$ of multiplicity $m$. Let $D$ be any disc of the complex plane centered at  $\xi_{\infty}$ containing no other element of the spectrum of $\bT_{\infty}.$ Then, for $\lambda$ large enough, $D$ contains exactly $m$ eigenvalues of $\bT_{\lambda}$ (repeated according to their respective multiplicities). Consequently, each eigenvalue $\xi_{\infty}>0$ of $\bT_{\infty}$ is a limit of eigenvalues $\xi$ of $\bT_{\lambda}$, as $\lambda$ goes to infinity.
\end{theorem}

In what follows, and only for simplify notations,  we will drop the subindex $\lambda$ to denote the solution operator.

\section{The mixed finite element method}
\label{sec:mixed_method}
The present section deals with the finite element approximation for the eigenvalue 
problem.  To do this task, we begin by introducing a regular family of triangulations of $\bar{\O}\subset\mathbb{R}^n$ denoted by $\{\CT_h\}_{h>0}$. Let $h_T$ the diameter of a triangle/tetrahedron $T\subset\CT_{h}$ and let us define $h:=\max\{h_T\,:\, T\in \CT_h\}$.
\subsection{The finite element spaces}
Given an integer $\ell\geq 0$ and a subset $D$ of $\mathbb{R}^n$, we denote by $\mathbb{P}_\ell(S)$ the space of polynomials of degree at most $\ell$ defined in $D$. We mention that, for tensorial fields we will define $\mathbf{P}_\ell(D):=[\mathbb{P}_\ell(D)]^{n\times n}$ and for vector fields $P_\ell(D):=[\mathbb{P}_\ell(D)]^n$.
With these ingredients at hand, for $k\geq 0$ we define the local Raviart-Thomas space of order $k$
 as follows  (see \cite{MR3097958})
 \begin{equation*}
 \mathbf{RT}_k(T)=[\mathbf{P}_k(T)]\oplus P_k(T)\boldsymbol{x},
 \end{equation*}
 where $\boldsymbol{x}\in\mathbb{R}^n$. With this local space, we define the global Raviart-Thomas space, which we denote by $\mathbb{RT}_k(\CT_h)$, as follows
 \begin{equation*}
 \mathbb{RT}_k(\CT_h):=\{\btau\in\mathbb{H}\,:\,(\tau_{i1},\cdots,\tau_{in})^{\texttt{t}}\in\mathbf{RT}_k(T)\,\,\forall i\in\{1,\ldots,n\},\,\,\forall T\in\CT_h\},
 \end{equation*}
 and we introduce the global space of piecewise polynomials of degree $\leq k$ defined by
 \begin{equation*}
 P_k(\CT_h):=\{\bv\in \L^2(\O)\,:\, v|_T\in P_k(T)\,\,\,\,\forall T\in\CT_h\}.
 \end{equation*}
 
 Also, we define
 \begin{equation*}
 \mathbb{H}_{h,0}:=\left\{ \btau_h\in\mathbb{RT}_k(\CT_h)\,\,:\,\,\int_{\O}\tr(\btau_h)=0  \right\},
 \end{equation*}
 and $\mathbf{Q}_h:=\mathbf{P}_k(\CT_h)$. 
 
Now we recall some well known approximation properties for the spaces defined above (see \cite{MR2009375} for instance).  Let $\bPi_h^k:\mathbb{H}^t (\O)\rightarrow \mathbb{RT}_k(\CT_h)$ be the Raviart-Thomas interpolation operator. For $t\in (0,1]$ and $\btau\in\mathbb{H}^t(\O)\cap\mathbb{H}(\bdiv;\O)$ the following error estimate holds true
 \begin{equation} \label{daniel1}
 \|\btau-\bPi_h^k\btau\|_{0,\O}\leq Ch^t \big(\|\btau\|_{t,\O}+\|\bdiv\btau\|_{0,\O}\big).
 \end{equation}
 
 Also, for $\btau\in\mathbb{H}^t(\O)$ with $t>1/2$, there holds
 \begin{equation}\label{daniel2}
 \|\btau-\bPi_h^k\btau\|_{0,\O}\leq Ch^{\min\{t,k+1\}} |\btau|_{t,\O}.
 \end{equation} 
 
 Let $\mathcal{P}_h^k:\L^2(\O)^n\rightarrow\mathbf{Q}_h$ be the $\L^2(\O)$-orthogonal projector. As a first property, we have the following commutative diagram
 \begin{equation}
 \label{eq:commutative}
 \bdiv(\bPi_h^k\btau)=\mathcal{P}_h^k(\bdiv\btau).
 \end{equation}
 
 If $\bv\in\H^t (\O)^{n}$ with $t>0$, there holds
 \begin{equation}\label{daniel3}
 \|\bv-\mathcal{P}_h^k\bv\|_{0,\O}\leq Ch^{\min\{t,k+1 \}} |\bv|_{t,\O}.
 \end{equation}
 
 Finally, for each $\btau\in\mathbb{H}^t(\O)$ such that $\bdiv\btau\in\H^t (\O)^{n}$, there holds
 \begin{equation} \label{daniel4}
 \|\bdiv(\btau-\bPi_h^k\btau)\|_{0,\O}\leq Ch^{\min\{t,k+1\}} |\bdiv\btau|_{t,\O}.
 \end{equation}

 \subsection{The discrete mixed eigenvalue problem}
Now we introduce the finite element discretization of \eqref{def:spectral_1}, which  reads as follows:
 Find $\kappa_h\in\mathbb{R}$ and $(\boldsymbol{\rho}_h,\bu_h)\in \mathbb{H}_{h,0}\times \mathbf{Q}_h $ such that
 
 \begin{equation}\label{def:spectral_1h}
\left\{
\begin{array}{rcll}
a(\boldsymbol{\rho}_h,\btau_h)+b(\btau_h,\bu_h) & = &0&  \forall\btau_h\in\mathbb{H}_{h,0}, \\
b(\boldsymbol{\rho}_h,\bv_h)& = & -\kappa_h(\bu_h,\bv_h)_{0,\O} &  \forall\bv_h\in \mathbf{Q}_h.
\end{array}
\right.
\end{equation}

We introduce the discrete kernel of $b(\cdot,\cdot)$ as follows
\begin{equation*}
\mathbb{V}_h:=\{\btau_h\in\mathbb{H}_{0,h}\,:\, b(\btau_h,\bv_h)=0\,\,\forall\bv_h\in\mathbf{Q}_h\}\subset\mathbb{V}.
\end{equation*}

It is clear that  $a(\cdot,\cdot)$ is elliptic in this space, i.e, there exists a positive constant $\overline{\alpha}$, independent of $h$,
such that
\begin{equation}
\label{eq:elipt_h}
a(\btau_h,\btau_h)\geq \overline{\alpha}\|\btau_h\|^2_{\bdiv,\O}\quad\forall \btau_h\in \mathbb{V}_h.
\end{equation}

Also, the following inf-sup condition holds (see \cite[Lemma 3.1]{MR3453481})
\begin{equation}
\label{eq_newinf-sup}
\displaystyle\sup_{\boldsymbol{0}\neq\btau_h\in \mathbb{H}_{0,h}}\frac{b(\btau_h,\bv_h)}{\|\btau_h\|_{\bdiv,\O}}\geq\beta\|\bv_h\|_{0,\O}\quad\forall\bv_h\in \mathbf{Q}_h,
\end{equation}
where $\beta>0$ is independent of $h$. 

Now, we introduce the discrete counterpart of $\bT$ 
\begin{align*}
\bT_{h}: \mathbf{Q}&\rightarrow\mathbf{Q}_h,\\
           \boldsymbol{f}&\mapsto \bT_h\boldsymbol{f}:= \widehat{\bu}_h, 
\end{align*}
where the pair $(\widehat{\boldsymbol{\rho}}_h, \widehat{\bu}_h)$ is the solution of the following source problem
 \begin{equation}\label{def:sourcel_1h}
\left\{
\begin{array}{rcll}
a(\widehat{\boldsymbol{\rho}}_h,\btau_h)+b(\btau_h,\widehat{\bu}_h) & = &0&  \forall\btau_h\in\mathbb{H}_{h,0}, \\
b(\widehat{\boldsymbol{\rho}}_h,\bv_h)& = & - (\boldsymbol{f},\bv_h)_{0,\O} &  \forall\bv_h\in \mathbf{Q}_h.
\end{array}
\right.
\end{equation}

Applying the Babu\^{s}ka-Brezzi theory, we have that the discrete operator $\bT_{h}$ is well defined and from \cite[Theorem 3.1]{MR3453481}, the following estimate holds 
\begin{equation*}
\|\widehat{\boldsymbol{\rho}}_h\|_{\bdiv,\O}+\|\widehat{\bu}_h\|_{0,\O}\leq C \|\boldsymbol{f}\|_{0,\O},
\end{equation*}
with $C>0$, independent of $h$ and $\lambda$.

\section{Convergence and Error estimates}
\label{sec:error}
We begin this section recalling some definitions of spectral theory. Let $\mathcal{X}$ be a generic Hilbert space and let $\bS$ be a linear bounded operator defined by $\bS:\mathcal{X}\rightarrow\mathcal{X}$. If $\boldsymbol{I}$ represents the identity operator, the spectrum of $\bS$ is defined by $\sp(\bS):=\{z\in\mathbb{C}:\,\,(z\boldsymbol{I}-\bS)\,\,\text{is not invertible} \}$ and the resolvent is its complement $\rho(\bS):=\mathbb{C}\setminus\sp(\bS)$. For any $z\in\rho(\bS)$, we define the resolvent operator of $\bS$ corresponding to $z$ by $R_z(\bS):=(z\boldsymbol{I}-\bS)^{-1}:\mathcal{X}\rightarrow\mathcal{X}$. 

Also, if $\mathcal{X}$ and $\mathcal{Y}$ are vectorial fields, we denote by $\mathcal{L}(\mathcal{X},\mathcal{Y})$ the space of all the linear and bounded operators acting from $\mathcal{X}$ to $\mathcal{Y}$.

The purpose of this section is to analyze the convergence of the mixed method and derive error estimates for the eigenvalues and eigenfunctions. Due the compactness of $\bT$, the convergence of the eigenvalues is derived from the classic theory of \cite{MR1115235}. 

The following result, which is a consequence of the convergence in norm between $\bT$ and $\bT_{h}$, reveals the convergence between the continuous and discrete solution operators.

\begin{lemma}
\label{lemma:P1}
Let $\boldsymbol{f}\in\mathbf{Q}$. There holds
\begin{equation*}
\|(\bT-\bT_h)\boldsymbol{f}\|_{0,\O}\leq C h^{\min\{s,k+1\}}\|\boldsymbol{f}\|_{0,\O},
\end{equation*}
where the positive constant $C$ is independent of $h$ and $\lambda$.
\end{lemma}
\begin{proof}
Let $\boldsymbol{f}\in\mathbf{Q}$. Then, since $\bT\boldsymbol{f}=\widehat{\bu}$ and $\bT_h\boldsymbol{f}=\widehat{\bu}_h$,  we have that 
\begin{equation}
\label{eq_firsteqP1}
\displaystyle \|(\bT-\bT_h)\boldsymbol{f}\|_{0,\O}=\|\widehat{\bu}-\widehat{\bu}_h\|_{0,\O}\leq \|\widehat{\bu}-\mathcal{P}_h^k\widehat{\bu}\|_{0,\O}+ \|\mathcal{P}_h^k\widehat{\bu}-\widehat{\bu}_{h}\|_{0,\O}.
\end{equation}

Set $\bv_{h}:=\mathcal{P}_h^k\widehat{\bu}-\widehat{\bu}_{h}\in \mathbf{Q}_h$ in \eqref{eq_newinf-sup}. Then
\begin{align*}
\|\mathcal{P}_h^k\widehat{\bu}-\widehat{\bu}_{h}\|_{0,\O}\leq \dfrac{1}{\beta}\displaystyle\sup_{\boldsymbol{0}\neq\btau_h\in \mathbb{H}_{0,h}}\frac{b(\btau_h,\mathcal{P}_h^k\widehat{\bu}-\widehat{\bu}_{h})}{\|\btau_h\|_{\bdiv,\O}}.
\end{align*}
Now, the fact that  $\btau_h\in \mathbb{H}_{0,h}$, then $\bdiv(\btau_{h})\in \mathbf{Q}_h$ and using that $\mathcal{P}_h^k$ is the $\L^2(\O)$-orthogonal projector, we have
\begin{multline*}
b(\btau_h,\mathcal{P}_h^k\widehat{\bu}-\widehat{\bu}_{h})=b(\btau_h,\widehat{\bu})-b(\btau_h,\widehat{\bu}_{h})\\
=a(\widehat{\boldsymbol{\rho}}_h,\btau_h)-a(\widehat{\boldsymbol{\rho}},\btau_h)
\leq C \|\widehat{\boldsymbol{\rho}}_h-\widehat{\boldsymbol{\rho}}\|_{0,\O}\|\btau_h\|_{0,\O},
\end{multline*}
where we have used the first equations of \eqref{def:sourcel_1} and  \eqref{def:sourcel_1h}.  
Therefore
\begin{align}
\label{eq_firsteqP2}
\|\mathcal{P}_h^k\widehat{\bu}-\widehat{\bu}_{h}\|_{0,\O}\leq C \|\widehat{\boldsymbol{\rho}}_h-\widehat{\boldsymbol{\rho}}\|_{0,\O}.
\end{align}

The following step is to bound $\|\widehat{\boldsymbol{\rho}}-\widehat{\boldsymbol{\rho}}_{h}\|_{0,\O}$. first note that 
\begin{align}
\label{eq_firsteqP3}
\|\widehat{\boldsymbol{\rho}}-\widehat{\boldsymbol{\rho}}_{h}\|_{0,\O}\leq \|\widehat{\boldsymbol{\rho}}-\bPi_h^k\widehat{\boldsymbol{\rho}}\|_{0,\O}+\|\bPi_h^k\widehat{\boldsymbol{\rho}}-\widehat{\boldsymbol{\rho}}_{h}\|_{0,\O}.
\end{align}
Now, using that $\left(\bPi_h^k\widehat{\boldsymbol{\rho}}-\widehat{\boldsymbol{\rho}}_{h}\right)\in \mathbb{H}_{0,h}$, \eqref{eq:commutative}, the second  equations of \eqref{def:sourcel_1}, and   \eqref{def:sourcel_1h}, we obtain the following 
\begin{equation*}
 \bdiv(\bPi_h^k\widehat{\boldsymbol{\rho}})=\mathcal{P}_h^k(\bdiv\widehat{\boldsymbol{\rho}})=\mathcal{P}_h^k(-\boldsymbol{f})=\bdiv\widehat{\boldsymbol{\rho}}_{h},
\end{equation*}
where it is straightforward that  $\bdiv\left(\bPi_h^k\widehat{\boldsymbol{\rho}}-\widehat{\boldsymbol{\rho}}_{h}\right)\in \mathbb{V}_h$.

Now, set  $\btau_{h}:=\bPi_h^k\widehat{\boldsymbol{\rho}}-\widehat{\boldsymbol{\rho}}_{h}$ in \eqref{eq:elipt_h}. Hence, 
\begin{align*}
\overline{\alpha}\|\bPi_h^k\widehat{\boldsymbol{\rho}}-\widehat{\boldsymbol{\rho}}_{h}\|_{0,\O}^{2}&=\overline{\alpha}\|\bPi_h^k\widehat{\boldsymbol{\rho}}-\widehat{\boldsymbol{\rho}}_{h}\|_{\bdiv,\O}^{2}\leq a(\bPi_h^k\widehat{\boldsymbol{\rho}},\bPi_h^k\widehat{\boldsymbol{\rho}}-\widehat{\boldsymbol{\rho}}_{h})-a(\widehat{\boldsymbol{\rho}}_{h},\bPi_h^k\widehat{\boldsymbol{\rho}}-\widehat{\boldsymbol{\rho}}_{h})\\
&=a(\bPi_h^k\widehat{\boldsymbol{\rho}},\bPi_h^k\widehat{\boldsymbol{\rho}}-\widehat{\boldsymbol{\rho}}_{h})-a(\widehat{\boldsymbol{\rho}},\bPi_h^k\widehat{\boldsymbol{\rho}}-\widehat{\boldsymbol{\rho}}_{h})-b(\bPi_h^k\widehat{\boldsymbol{\rho}}-\widehat{\boldsymbol{\rho}}_{h},\widehat{\bu})\\
&=a(\bPi_h^k\widehat{\boldsymbol{\rho}}-\widehat{\boldsymbol{\rho}},\bPi_h^k\widehat{\boldsymbol{\rho}}-\widehat{\boldsymbol{\rho}}_{h}) \\
&\leq C \|\bPi_h^k\widehat{\boldsymbol{\rho}}-\widehat{\boldsymbol{\rho}}\|_{0,\O}\|\bPi_h^k\widehat{\boldsymbol{\rho}}-\widehat{\boldsymbol{\rho}}_{h}\|_{0,\O}.
\end{align*}

These calculations imply that 
\begin{align}
\label{eq_firsteqP4}
\|\bPi_h^k\widehat{\boldsymbol{\rho}}-\widehat{\boldsymbol{\rho}}_{h}\|_{0,\O}
\leq C \|\bPi_h^k\widehat{\boldsymbol{\rho}}-\widehat{\boldsymbol{\rho}}\|_{0,\O},
\end{align}
and, invoking \eqref{eq_firsteqP1}, \eqref{eq_firsteqP2}, \eqref{eq_firsteqP3} and \eqref{eq_firsteqP4}, we have
\begin{equation*}
\|(\bT-\bT_h)\boldsymbol{f}\|_{0,\O} \leq C \left(\|\widehat{\bu}-\mathcal{P}_h^k\widehat{\bu}\|_{0,\O}+\|\bPi_h^k\widehat{\boldsymbol{\rho}}-\widehat{\boldsymbol{\rho}}\|_{0,\O}\right).
\end{equation*}

Finally, the proof is concluded  from the above estimate, \eqref{daniel2}, \eqref{daniel3}  and  \eqref{eq_cotasupfuente}.
\end{proof}

As a direct consequence of Lemma~\ref{lemma:P1}, standard results about
spectral approximation (see \cite{MR0203473}, for instance) show that isolated
parts of $\sp(\bT)$ are approximated by isolated parts of $\sp(\bT_h)$. More
precisely, let $\xi\in(0,1)$ be an isolated eigenvalue of $T$ with
multiplicity $m$ and let $\mathcal{E}$ be its associated eigenspace. Then, there
exist $m$ eigenvalues $\xi^{(1)}_h,\dots,\xi^{(m)}_h$ of $\bT_h$ (repeated
according to their respective multiplicities) which converge to $\xi$.

%
 Now we are in position to establish that our method does not introduce spurious
eigenvalues, which is stated in the following result (see \cite{MR0203473} for instance).
\begin{theorem}[Spurious free]
\label{thm:spurious_free}
Let $V\subset\mathbb{C}$ be an open set containing $\sp(\bT)$. Then, there exists $h_0>0$ such that $\sp(\bT_h)\subset V$ for all $h<h_0$.
\end{theorem}

Let us  recall the definition of the resolvent operator of $\bT$ and $\bT_h$ respectively:
\begin{gather*}
	R_z(\bT):=(z\bI-\bT)^{-1}\,:\, \mathbf{Q} \to \mathbf{Q}\,, \quad z\in\mathbb{C}\setminus \sp(\bT), \\
	R_z(\bT_{h}):=(z\bI-\bT_h)^{-1}\,:\, \mathbf{Q}_h \to \mathbf{Q}_h\,, \quad z\in\mathbb{C}\setminus\sp(\bT_h) .
\end{gather*}

We also  invoke the following result for the resolvent of $\bT$.
\begin{prop}\label{prop:bounded_resolv}
If $z\notin\sp(\bT_{\lambda})$, then there exists a positive constant $C$, independent of $\lambda$ and $z$ such that
\begin{equation*}
\|(z\bI-\bT_{\lambda})\bu\|_{0,\O}\geq C\dist(z,\sp(\bT_{\lambda}))\|\bu \|_{0,\O},
\end{equation*}
where $\dist(z,\sp(\bT ))$ represents the distance between $z$ and the spectrum of $\bT$ in the complex plane, which in principle depends on $\lambda$. 
\end{prop}
\begin{proof}
See \cite[Proposition 2.4]{MR3376135}.
\end{proof}

Now we prove the analogous result presented above, but for the resolvent of the discrete solution operator:
  \begin{lemma}
 \label{thm:bounded_resolvent}
 Let $F\subset\rho(\bT)$ be closed. Then, there exist
 positive constants $C$ and $h_0$, independent of $h$, such that for $h<h_0$
 \begin{equation*}
 \displaystyle\|(z\bI-\bT_h)^{-1}\boldsymbol{f}\|_{0,\O}\leq C\|\boldsymbol{f}\|_{0,\O}\qquad\forall z\in F.
 \end{equation*}
 \end{lemma}
\begin{proof}
	Let $\boldsymbol{f} \in \bQ$. From Proposition \ref{prop:bounded_resolv}, there exists $C>0$, independent of $\lambda$ and $z$ such that
	\[
		\| (z\bI-\bT)\boldsymbol{f} \|_{0,\O}\geq C\,\dist( z, \sp(\bT) )\, \| \boldsymbol{f} \|_{0,\O} \quad \forall z\in F.
	\]
	Then, we have
	\begin{align*}
		\| (z\bI-\bT_h)\boldsymbol{f} \|_{0,\O} &= \| (z\bI-\bT)\boldsymbol{f}+(\bT-\bT_h)\boldsymbol{f} \|_{0,\O} \\
		                         & \geq \| (z\bI-\bT)\boldsymbol{f} \|_{0,\O} - \| (\bT-\bT_h)\boldsymbol{f} \|_{0,\O} \\
		                         & \geq C\,\dist( z, \sp(\bT) )\| \boldsymbol{f} \|_{0,\O}  -\| (\bT-\bT_h)\boldsymbol{f} \|_{0,\O} \\
		                         & \geq  \widetilde{C}\,  \| \boldsymbol{f} \|_{0,\O}.
	\end{align*}
	Then, the result follows from the previous inequality and Lemma \ref{lemma:P1}, where 
	\[
	 	\widetilde{C}:= C\,\dist( z, \sp(\bT)).
	\]
\end{proof}

Our next task is to derive error estimates for the eigenvalues and eigenfunctions. 
Let $\boldsymbol{E}:\mathbf{Q}\rightarrow\mathbf{Q}$ be the spectral projector of $\bT$ corresponding to the isolated 
eigenvalue $\xi$, namely
\begin{equation*}
\displaystyle \boldsymbol{E}:=\frac{1}{2\pi i}\int_{\gamma} R_{z}(\bT)dz.
\end{equation*}
On the other, we define $\boldsymbol{E}_h:\mathbf{Q}\rightarrow\mathbf{Q}$ as the spectral projector of $\bT_h$ corresponding to the isolated 
eigenvalue $\xi_h$, namely
\begin{equation*}
\displaystyle \boldsymbol{E}_h:=\frac{1}{2\pi i}\int_{\gamma} R_{z}(\boldsymbol{T}_h)dz.
\end{equation*}
%

Let $\kappa$ be an isolated eigenvalue of $\bT$. We define the following distance
\begin{equation*}
\texttt{d}_\kappa:=\frac{1}{2}\dist\left(\kappa,\sp(\bT)\setminus\{\kappa\}\right).
\end{equation*}

With this distance at hand, we define the disk centered in $\kappa$ and boundary $\gamma$ as follows
\begin{equation*}
D_\kappa:=\{z\in\mathbb{C}:\,\,|z-\kappa|\leq \texttt{d}_\kappa\}.
\end{equation*}
We observe that the disk defined above satisfies $D_\kappa\cap\sp(\bT)=\{\kappa\}$.

\begin{lemma}
\label{lemma:spectral_projectors}
Let $\boldsymbol{f}\in \bQ$. There exist constants $C>0$ and $h_{0}>0$ such that, for all $h<h_{0}$,
\begin{equation*}
	\|  (\boldsymbol{E}-\boldsymbol{E}_{h})\boldsymbol{f} \|_{0,\O}\leq \dfrac{C}{\texttt{d}_{\kappa}}\|(\bT-\bT_h)\boldsymbol{f}\|_{0,\Omega}\leq\dfrac{C}{\texttt{d}_{\kappa}}\,h^{ \min\{ s,k+1 \} } \| \boldsymbol{f} \|_{0,\O}.
\end{equation*}

\begin{proof} The proof follows by repeating the same arguments of those in \cite[Lemma 5.3]{MR3962898}.

\end{proof}

\end{lemma}

We recall the definition of the \textit{gap} $\hdel$ between two closed
subspaces $\CM$ and $\CN$ of $\L^2(\O)$:
$$
\hdel(\CM,\CN)
:=\max\big\{\delta(\CM,\CN),\delta(\CN,\CM)\big\},
$$
where
$$
\delta(\CM,\CN)
:=\sup_{x\in\CM:\ \left\|x\right\|_{0,\O}=1}
\left(\inf_{y\in\CN}\left\|x-y\right\|_{0,\O}\right).
$$

\begin{theorem}
\label{millar2015}
There exists strictly positive constant C, such that 
\begin{equation*}
	\widehat{\delta} ( \mathcal{E}, \mathcal{E}_{h} )\leq \dfrac{C}{d_{\kappa}}\,h^{ \min\{ s,k+1 \} } \quad \mbox{and} \quad | \xi-\xi_{h}(i) | \leq \dfrac{C}{d_{\kappa}}\,h^{ \min\{s,k+1 \} },
\end{equation*}
where $\xi_{h}(1), \ldots , \xi_{h}(m)$ are the eigenvalues of $\bT_{h}$.
\begin{proof}
	As consequence of Lemma \ref{lemma:P1}, the convergence in norm to $\bT -\bT_{h}$ as $h$ goes to zero. Then, the proof follows as a direct consequence of Lemma \ref{lemma:spectral_projectors} and \cite[Theorems 7.3]{MR1115235}.   
\end{proof}
\end{theorem}

%
%
%

The next
step is to show an optimal order estimate for this term.

As is customary in eigenvalue problems, we can improve the simple order obtained in Theorem \ref{millar2015} for the eigenvalues, which is stated in  next result.
\begin{theorem}
\label{thm:double}
There exists a strictly positive constant $h_0$ such that, for $h<h_0$ there holds
\begin{equation*}
|\kappa-\kappa_h| \leq \dfrac{C}{d_{\kappa}} h^{2\min\{s,k+1\}},
\end{equation*}
where the positive constant $C$ is independent of $h$ and $\lambda$.
\end{theorem}
\begin{proof}
	Let $(\kappa , (\boldsymbol{\rho}, \boldsymbol{u}))$ and $(\kappa_{h} , (\boldsymbol{\rho}_{h}, \boldsymbol{u}_{h}))$ be the solutions of problems \eqref{def:spectral_1} and \eqref{def:spectral_1h} respectively, with $\| \boldsymbol{u} \|_{0,\Omega}=\| \boldsymbol{u}_{h} \|_{0,\Omega}=1$. 
	
 For the proof, we use the following estimate that  is proved for mixed
methods in general (see \cite{MR4080229,MR1722056} for more details).
	\begin{equation*}
		\kappa-\kappa_{h}=\| \boldsymbol{\rho}-\boldsymbol{\rho}_h \|_{0,\O} +\kappa_{h}\| \boldsymbol{u}-\boldsymbol{u}_h \|_{0,\O}.
		\end{equation*}
%

	According to Remark \ref{daniel0}, in addition to the approximation properties \eqref{daniel1}, \eqref{daniel2}, \eqref{eq:commutative}, \eqref{daniel3}, \eqref{daniel4} and Theorem \ref{millar2015}, we obtain 
	\begin{equation*}
	\|\bu-\bu_h\|_{0,\O}\leq \dfrac{C}{d_{\kappa}} h^{  \min\{s,k+1\}} \quad \mbox{and}\quad \| \boldsymbol{\rho}-\boldsymbol{\rho}_h \|_{0,\O}\leq \dfrac{C}{d_{\kappa}} h^{  \min\{s,k+1\}},  
	\end{equation*}
	where the positive constant $C$ is uniform on $h$ and $\lambda$.

	This concludes the proof.
\end{proof}

\section{Numerical experiments}
\label{sec:numerics}
The aim of this section is to confirm, computationally, that the proposed method works correctly and delivers an accurate approximation of the spectrum 
of $\bT$, reinforcing  the theoretical results of our study. The reported results have been obtained with a FEniCS code \cite{MR3618064}, considering the meshes provided by this software. 

For our experiments we consider as Young's modulus $E=1$. The Poisson ratio $\nu$ will take different values. It is well known that the Lam\'e constant
$\lambda$ blows up when  $\nu=1/2$. Is for this reason that we are interested in the performance of the method in the case limit case $\lambda=\infty$.
The Lam\'e coefficients are defined by 
$$\lambda:=\frac{E\nu}{(1+\nu)(1-2\nu)}\quad\text{and}\quad\mu:=\frac{E}{2(1+\nu)}.$$

We compute the eigenvalues and eigenfunctions considering different polynomial degrees in the unitary square,  the unitary cube, the unitary  circle and the classic  L-shaped domain. 
We also report in the following  tables  an estimate  of the order of
convergence $\alpha$ and, in the last column,
more accurate values of the vibration frequencies $\omega_{extr}:=\sqrt{\kappa_{extr}}$,
extrapolated from the computed ones by means of a
least-squares fitting of the model
$$\omega_{hi}\approx\omega_{i}+C_ih^{\alpha_i},$$
that has been done for each vibration mode separately. The fitted parameters $\omega_i$ and $\alpha_i$ are the reported
extrapolated vibration frequency $\omega_{extr}$ and  estimated order of
convergence, respectively.


\subsection{Unitary square.} The considered geometry for this test is $\O:=(0,1)^2$. Since we are interested in the stability
for different values of $\lambda$, we compute the eigenvalues for $\nu=0.35, 0.49, 0.5$. Clearly in the limit 
case $\nu=0.5$, the Lam\'e constant $\lambda=\infty$,  leading to a modification on  \eqref{eq:identity_a}
as follows
\begin{equation*}
\displaystyle a(\bxi,\btau):=\frac{1}{\mu}\int_{\Omega}\bxi^{\texttt{d}}:\btau^{\texttt{d}}\quad\btau \in\mathbb{H}.
\end{equation*}

In this test, we consider meshes like the presented in Figure \ref{meshes_square}.
\begin{figure}[H]
	\begin{center}
		\begin{minipage}{11cm}
			\centering\includegraphics[height=4.6cm, width=3.3cm]{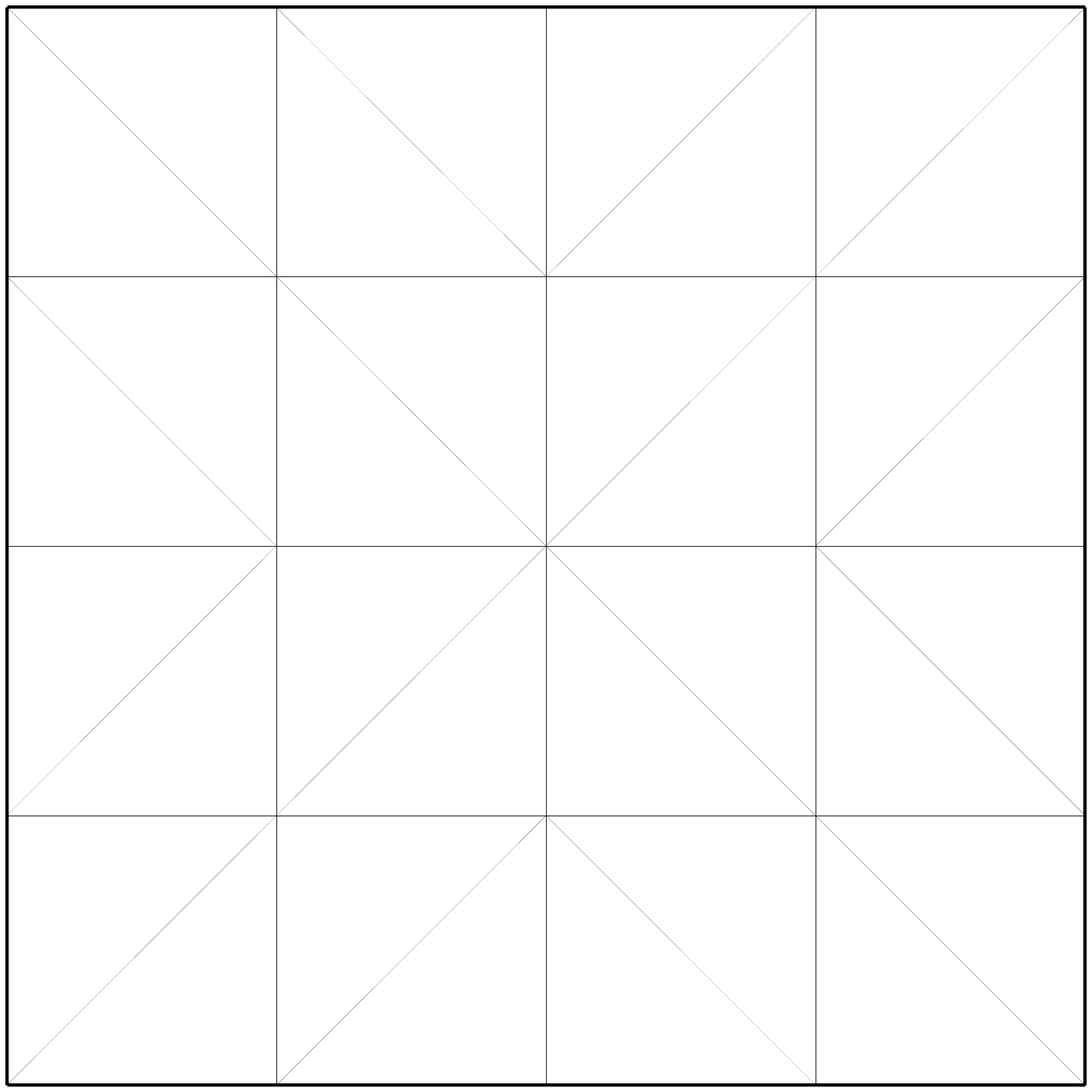}
			\centering\includegraphics[height=4.6cm, width=3.3cm]{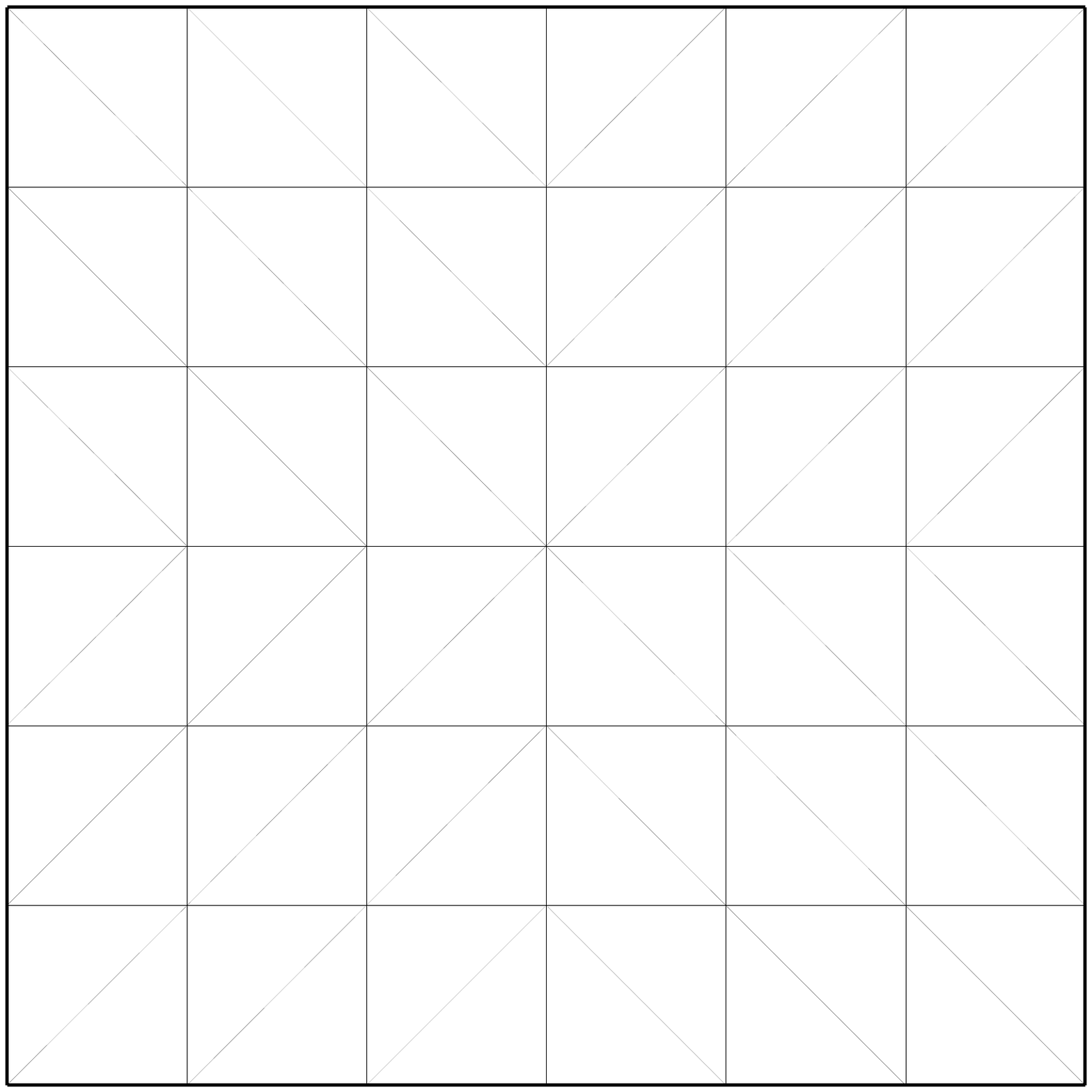}
                   \end{minipage}
		\caption{Examples of the meshes used in the unit square. The left figure represents a mesh for $N=4$ and the right one for $N=6$.}
		\label{meshes_square}
	\end{center}
\end{figure}
 In the following tables, the parameter $N$ will represent the refinement level of the meshes  and it is chosen as 
 the number of subdivisions in the abscissa.
We report in Table \ref{table:k=0_cuadrado} the lowest vibration frequencies for $k=0$ and different Poisson's ratio in the unitary square. The table also includes the estimated orders of convergence. The accurate values extrapolated are also reported in the last column to allow for comparison.
\begin{table}[H]

\begin{center}
\begin{tabular}{c |c c c c |c| c }
\toprule
 $\nu $        & $N=40$             &  $N=50$         &   $N=60$         & $N=70$ & $\alpha$& $\omega_{extr}$  \\ 
 \midrule
					 & 4.19038  & 4.19134&  4.19187&   4.19219 &  1.94&  4.19311\\
                                           & 4.19038  & 4.19134&  4.19187&   4.19219 &  1.94&  4.19311\\
 \multirow{2}{1cm}{0.35}     & 4.37189 & 4.37199&  4.37205&   4.37208 &  1.95&  4.37217\\
  					 & 5.92825  & 5.92995&  5.93089&   5.93147 &  1.89&  5.93318\\

          \hline

  					   &4.18710 &  4.18820&  4.18832&  4.18839 &  1.95 &  4.18858\\
  					   &5.51379 &  5.51516&  5.51590&  5.51634 &  2.00 &  5.51758\\
 \multirow{2}{1cm}{0.49}       &5.51379 &  5.51516&  5.51590&  5.51634 &  2.00 &  5.51758\\
					   &6.53985 &  6.54111&   6.54180 & 6.54221  & 1.99  & 6.54337\\

\hline

					&4.17650 &  4.17672 &  4.17683 &  4.17691 &  1.95 &  4.17711\\
					&5.53828 &  5.53944 &  5.54007 &  5.54044 &  2.00 &  5.54149\\
\multirow{2}{1cm}{0.5}       &5.53828 &  5.53944 &  5.54007 &  5.54044 &  2.00 &  5.54149\\
					&6.53394 &  6.53515 &  6.53581 &  6.53621 & 1.99  &  6.53732\\

\bottomrule             
\end{tabular}
\end{center}
\caption{Computed lowest vibration frequencies for $k=0$ and different Poisson's ratio in the unitary square.}
\label{table:k=0_cuadrado}
\end{table}
We remark that this table presents a clear quadratic order of convergence as is expected according to Theorem \ref{thm:double}.

In the following experiment, we will prove the accuracy of the method for other polynomial degrees.
In particular, and for simplicity, we will consider  $\nu=0.49$ and $k=0,1,2$. 
\begin{table}[H]
\footnotesize
\singlespacing
\begin{center}
\begin{tabular}{c |c c c c |c| c }
\toprule
 $k $        & $N=20$             &  $N=30$         &   $N=40$         & $N=50$ & $\alpha$& $\omega_{extr}$  \\ 
 \midrule
					 &4.18639&   4.18757&   4.18710&   4.18820&   1.87 & 4.18860 \\
                                           &5.50235&   5.51084&   5.51379&   5.51516&   2.01 & 5.51757\\
 \multirow{2}{0.5cm}{0}      & 5.50235&   5.51084&   5.51379&  5.51515 &  2.01 & 5.51757\\
  					 &6.52943&   6.53714&   6.53985&   6.54111&  1.99 &  6.54336 \\

          \hline

  					   &4.18857&   4.18857&   4.18858&   4.18858&   3.50 &  4.18858\\
  					   &5.51760&   5.51758&   5.51758&   5.51758&   4.81 &  5.51758\\
 \multirow{2}{0.5cm}{1}        &5.51760&    5.51758 &  5.51758 &  5.51758 &  4.81  & 5.51758\\
					   &6.54340&   6.54337&   6.54336&   6.54336&   5.27 &  6.54336 \\

\hline

					&4.18858&   4.18858 &  4.18858&   4.18858 &  5.79 &  4.18858\\
					&5.51758&   5.51758 &  5.51758&   5.51758 &  3.55 &  5.51758\\
\multirow{2}{0.5cm}{2}       &5.51758&   5.51758 &  5.51758&   5.51758 &  3.55 &  5.51758\\
					&6.54336&   6.54336 &  6.54336&   6.54336 &  5.33 &  6.54336\\

\bottomrule             
\end{tabular}
\end{center}
\caption{Computed lowest vibration frequencies for $k=0,1,2$ and $\nu=0.49$ in the unitary square.}
\label{table:ks_cuadrado}
\end{table}

\begin{table}[H]
\footnotesize
\singlespacing
\begin{center}
\begin{tabular}{c |c c c c |c| c }
\toprule
 $k $        & $N=5$             &  $N=10$         &   $N=20$         & $N=40$ & $\alpha$& $\omega_{extr}$  \\ 
 \midrule
                                         & 4.13253 & 4.17434 & 4.18467  & 4.18763 &  1.98 &  4.18843\\
                                         &5.29856  & 5.47666  & 5.50805  & 5.51501 &  2.46  & 5.51589\\
 \multirow{2}{0.5cm}{0}    &5.32537  & 5.48047  & 5.50814  & 5.51513 &  2.42  & 5.51572\\
                                        &6.16802  & 6.47995  & 6.52542  & 6.53839 &  2.66  & 6.53763\\

          \hline
                                           &4.18723  & 4.18847 &  4.18857  & 4.18858  & 3.63  & 4.18858\\
                                         & 5.51759 &  5.51757 &  5.51758 &  5.51758  & 10.00&   5.51758\\
 \multirow{2}{0.5cm}{1}     & 5.51978  & 5.51768 &  5.51758 &  5.51758  & 4.49  & 5.51758\\
                                          & 6.53153  & 6.54286  & 6.54332 &  6.54336  & 4.59  & 6.54335\\

\hline

                                       &4.18845  & 4.18857  & 4.18858  & 4.18858  & 5.86  & 4.18858\\
                                      & 5.51726 &  5.51758 &  5.51758 &  5.51758 &  6.25 &  5.51758\\
 \multirow{2}{0.5cm}{2} &  5.51755  & 5.51758  & 5.51758 &  5.51758 &  3.66 &  5.51758\\
                                     & 6.54304  & 6.54335  & 6.54336 &  6.54336 &  5.34 &  6.54336\\

\bottomrule             
\end{tabular}
\end{center}
\caption{Computed lowest vibration frequencies for $k=0,1,2$ and $\nu=0.49$ in the unitary square.}
\label{table:ks_cuadradocul}
\end{table}

It is clear that for $k>0$ the convergence order of the eigenfrequencies is $\mathcal{O}(h^{2(k+1)})$ according to Theorem \ref{thm:double}. We observe from Tables \ref{table:ks_cuadrado} and \ref{table:ks_cuadradocul} that there are vibration frequencies that converge with optimal order, however, some orders are affected when $k>0$ and the meshes are sufficiently refined. For example, in Table \ref{table:ks_cuadrado}, a deterioration in the order of convergence is observed, due to the fact that the vibration frequencies obtained are very close to the extrapolated vibration frequencies. Also,  it is observed in Table \ref{table:ks_cuadradocul} that if coarse meshes are used the optimal order is recovered as is expected.

To better visualize the errors, we concentrate on Table \ref{table:ks_cuadradocul}. For this purpose, in Figure \ref{FIG:erroresK} we show the relative errors   for the vibration frequencies $e_{\kappa_{hi}}$ where  $i\in\{1,2,3,4\}$, for  different polynomial degrees $k$, presented in Table \ref{table:ks_cuadradocul}.  In Figure \ref{FIG:erroresK} we present  lines of slopes $2(k+1)$ which we have obtained  using the extrapolated values obtained in Table \ref{table:ks_cuadradocul} as exact eigenvalues. 

Thus $e_{\kappa_{hi}}$ is defined by
$$e_{\kappa_{hi}}:=\dfrac{|\omega_{hi}-\omega_{extr,i}|}{|\omega_{extr,i}|},\qquad i=\{1,2,3,4\}.$$  
\begin{figure}[H]
	\begin{center}
		\begin{minipage}{14cm}
			\hspace{-1.1 cm}\centering\includegraphics[height=5.4cm, width=4.4cm]{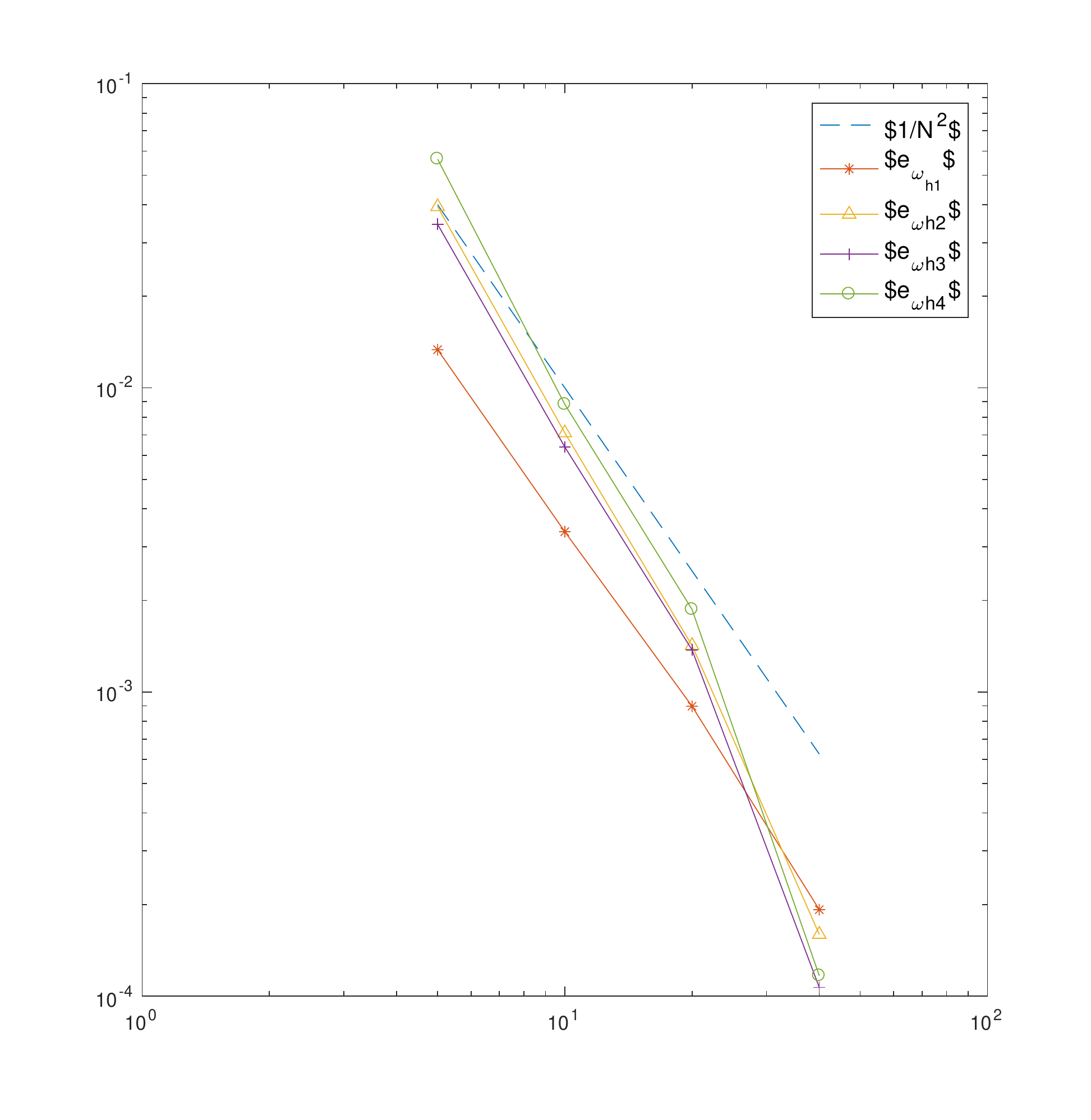}
			\centering\includegraphics[height=5.4cm, width=4.4cm]{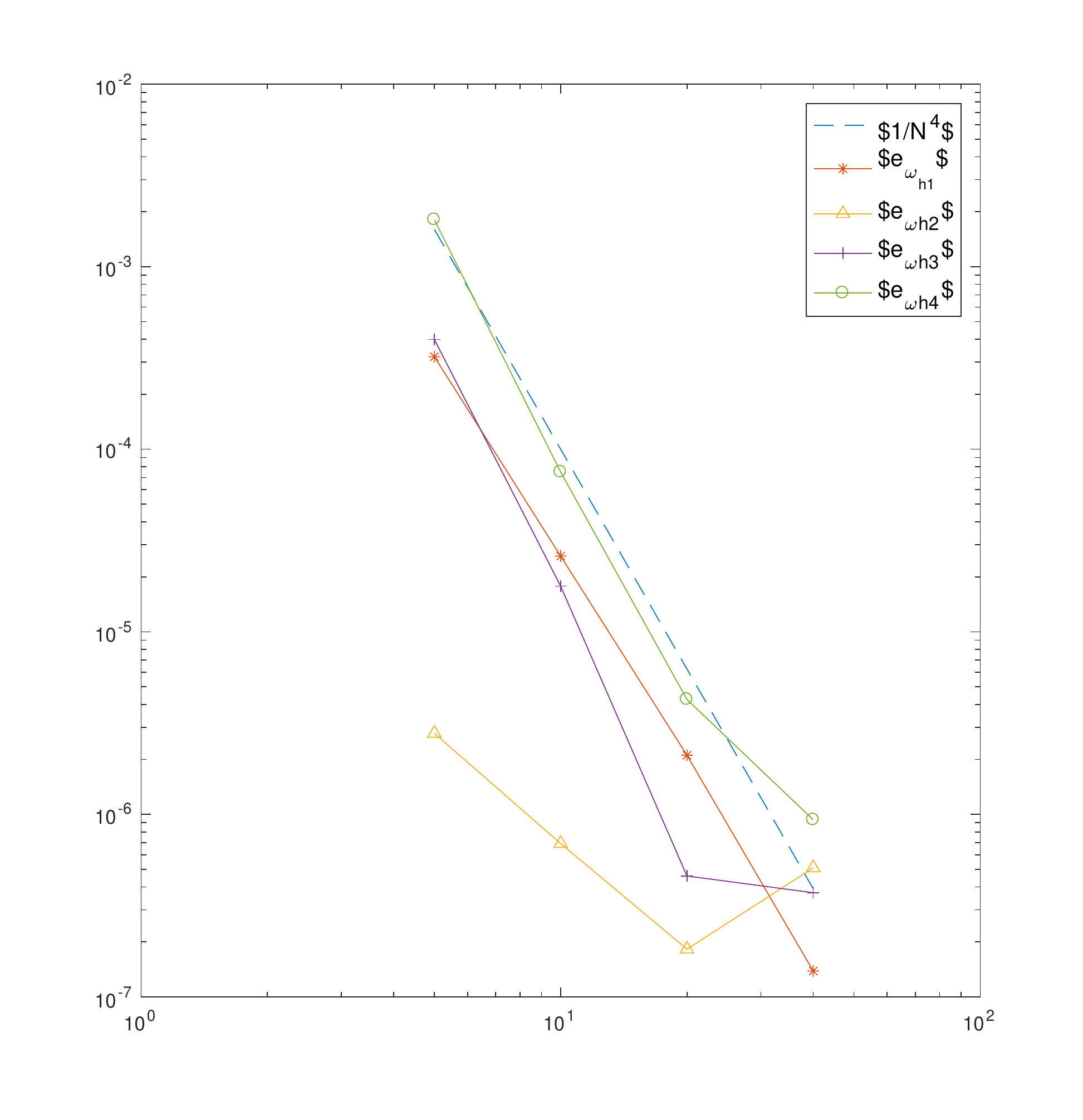}
			\centering\includegraphics[height=5.4cm, width=4.4cm]{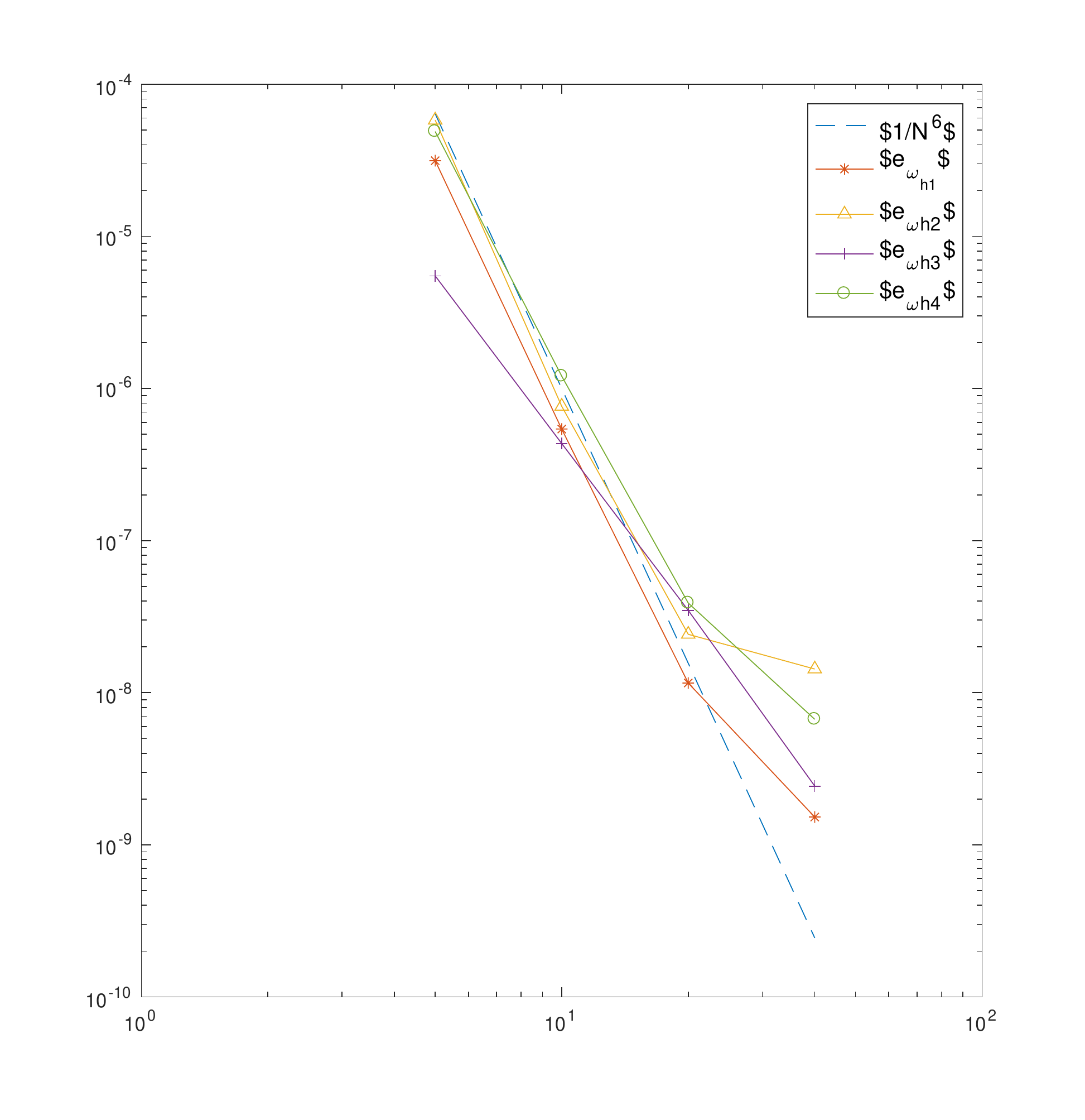}
                   \end{minipage}
		\caption{Relative errors for Table \ref{table:ks_cuadradocul} together with the expected optimal orders, for $k=0$ (left), $k=1$ (middle) and $k=2$ (right).}
		\label{FIG:erroresK}
	\end{center}
\end{figure}
From Figure \ref{FIG:erroresK} we confirm that for $k>0$ the discrete eigenvalues obtained are very similar to the extrapolated eigenvalue, which generates the precision errors in the correct convergence  obtained.

We present in Figure \ref{FIG:squares} plots of the first and third eigenfunctions of the spectral problem in the presented configuration. The colors
represent the magnitude of the displacement $\bu$ of the elastic structure.


\begin{figure}[H]
	\begin{center}
		\begin{minipage}{14cm}
			\hspace{-1.0 cm}\centering\includegraphics[height=4.0cm, width=4.0cm]{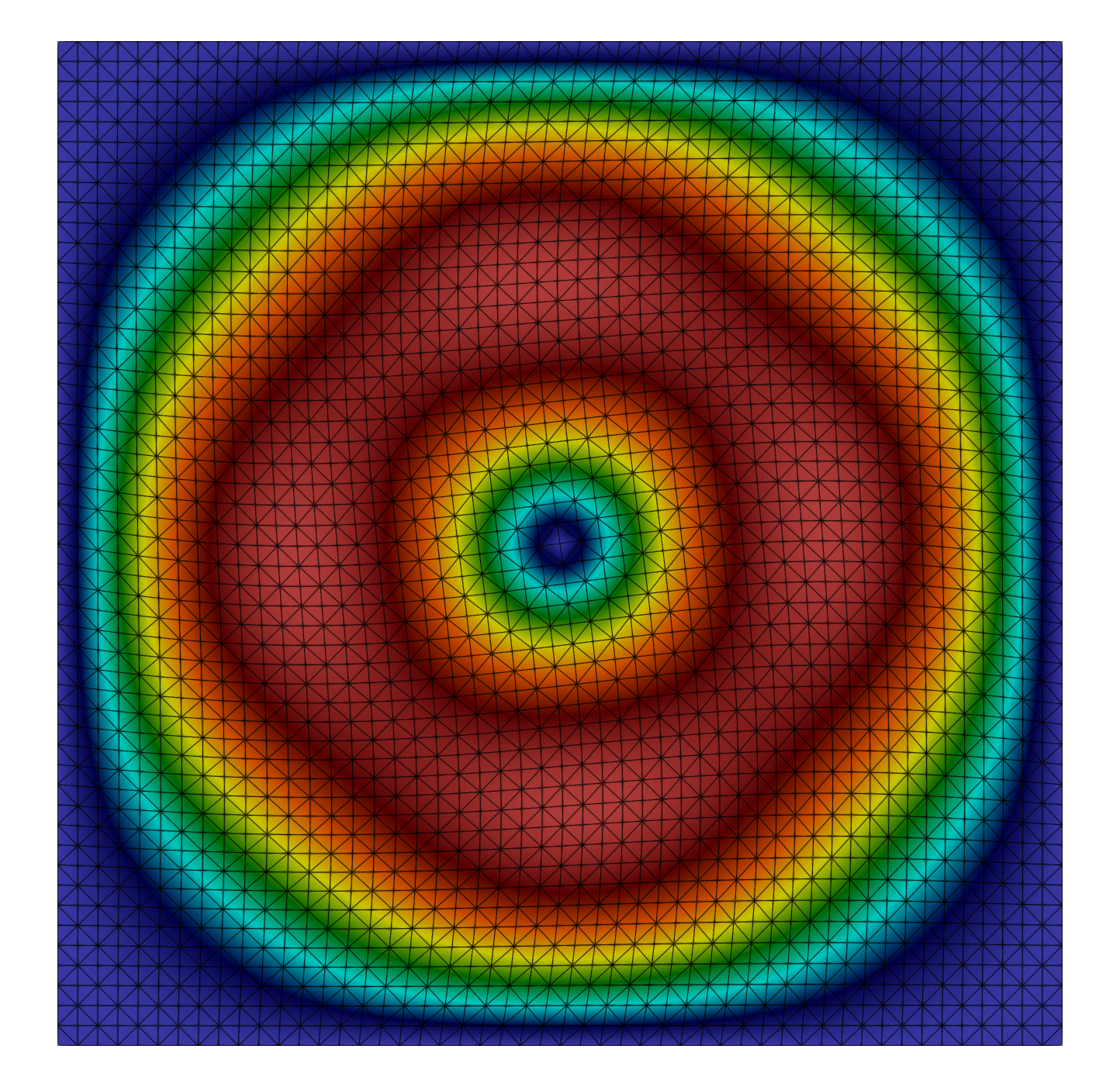}
			\centering\includegraphics[height=4.0cm, width=4.0cm]{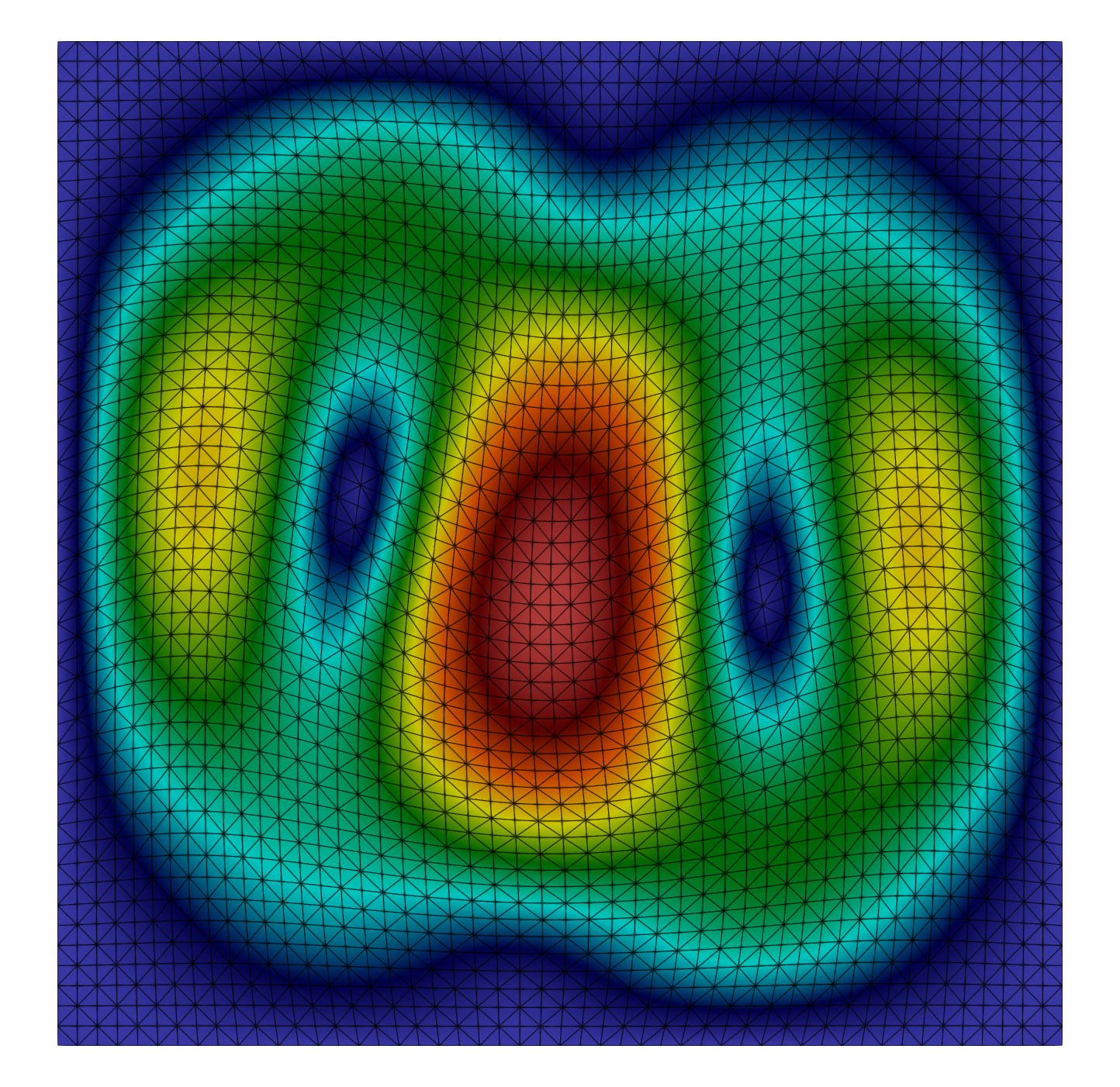}
			\centering\includegraphics[height=4.0cm, width=4.0cm]{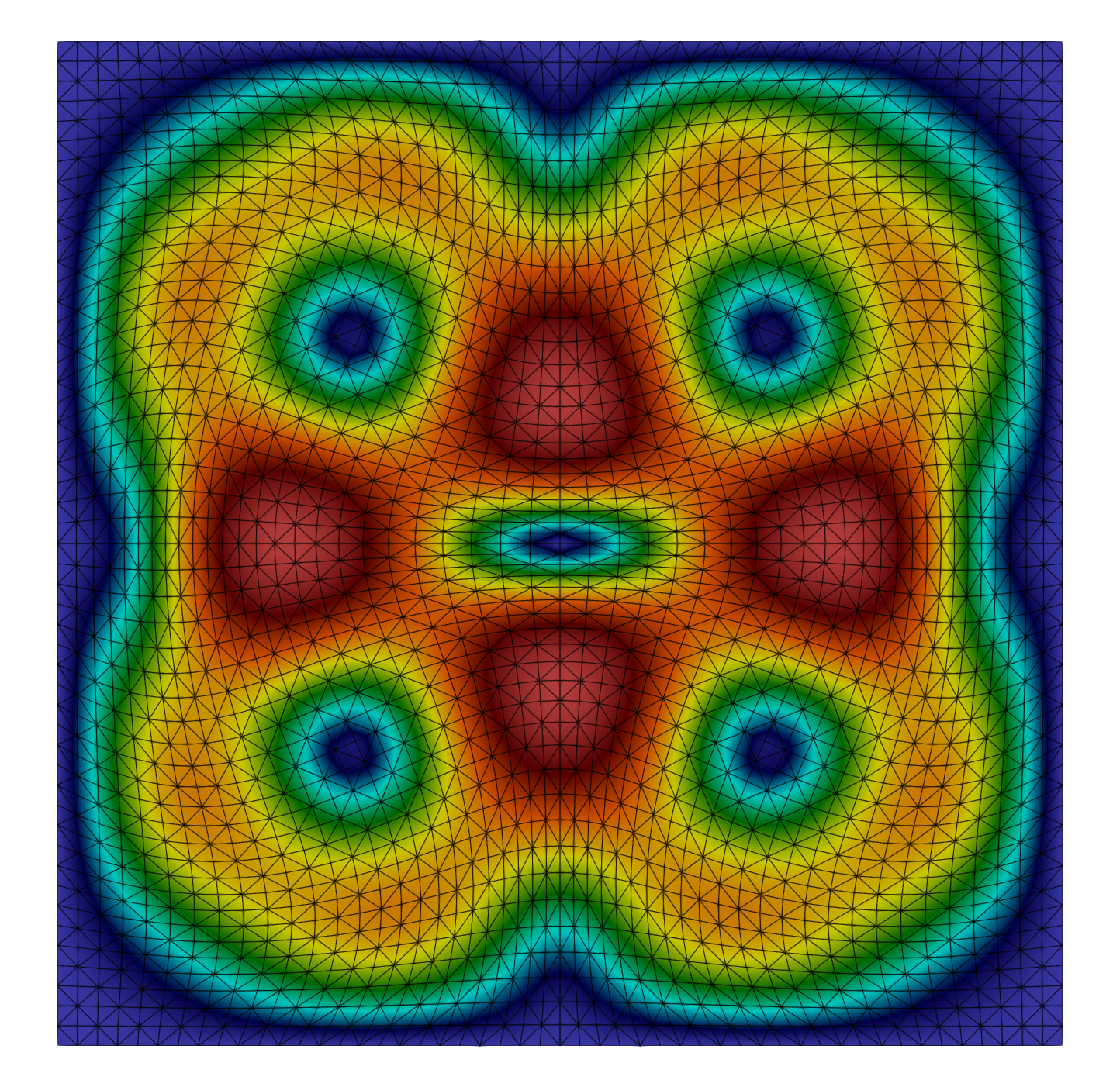}
			\end{minipage}
		\caption{Eigenfunctions corresponding to the first (left), second (middle) and fourth (right) eigenvalues with $\nu=0.49$, $N=50$ and $k=1$.}
		\label{FIG:squares}
	\end{center}
\end{figure}

\subsection{Unitary cube} In the following test we consider a three dimensional domain. For simplicity, we have
chosen the unitary cube $\O:=(0,1)^3$ and the lowest order finite element spaces (i.e. $k=0$). 
The meshes for this tests consist in regular tetrahedrons  and $N$, which we consider as refinement level, corresponds to the number of tetrahedrons in the plane $XY$, with partitions respect to the $X$ axis.
\begin{figure}[H]
	\begin{center}
		\begin{minipage}{14cm}
			\hspace{-1.1 cm}\centering\includegraphics[height=4.4cm, width=4.4cm]{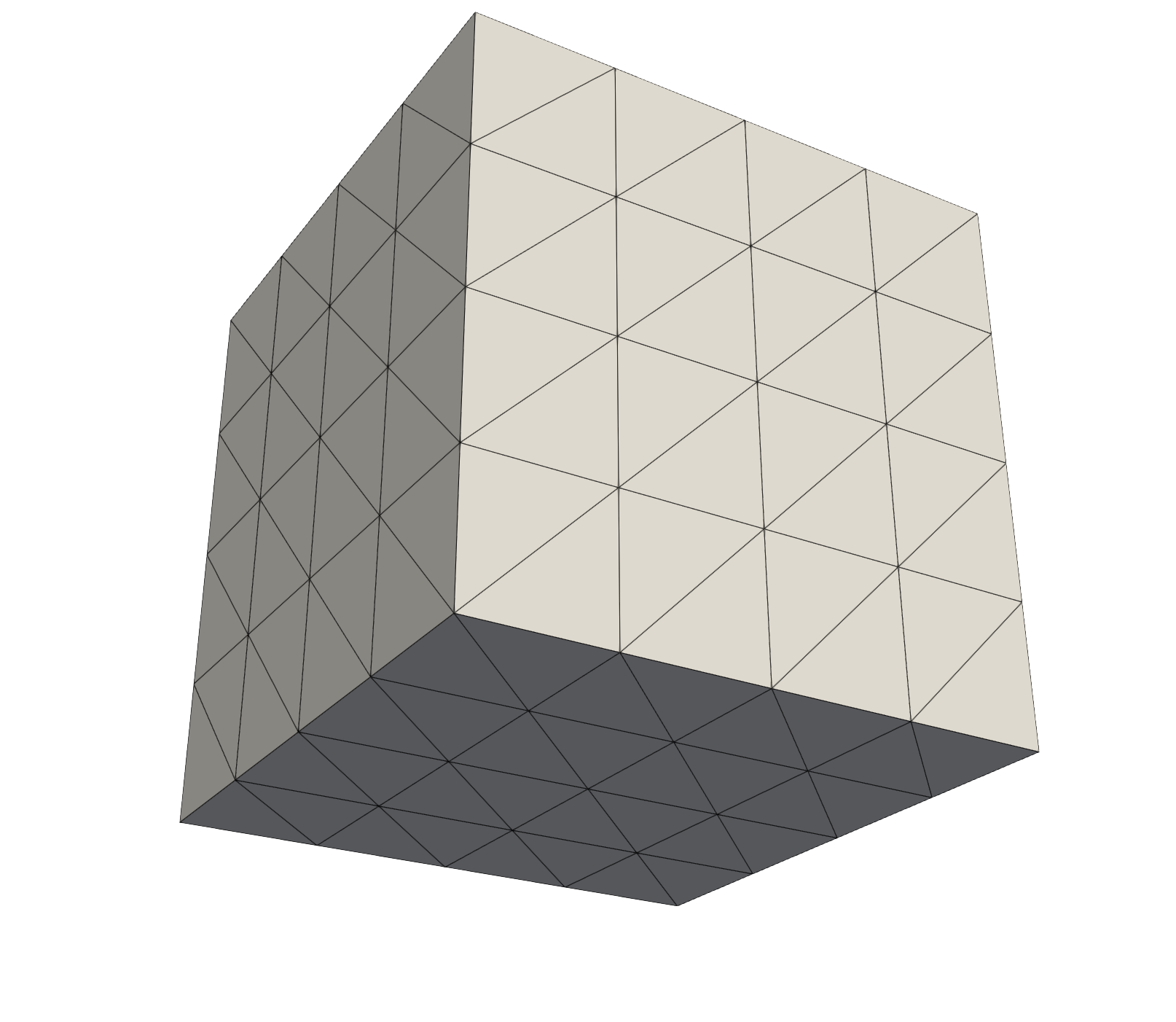}
			\centering\includegraphics[height=4.4cm, width=4.4cm]{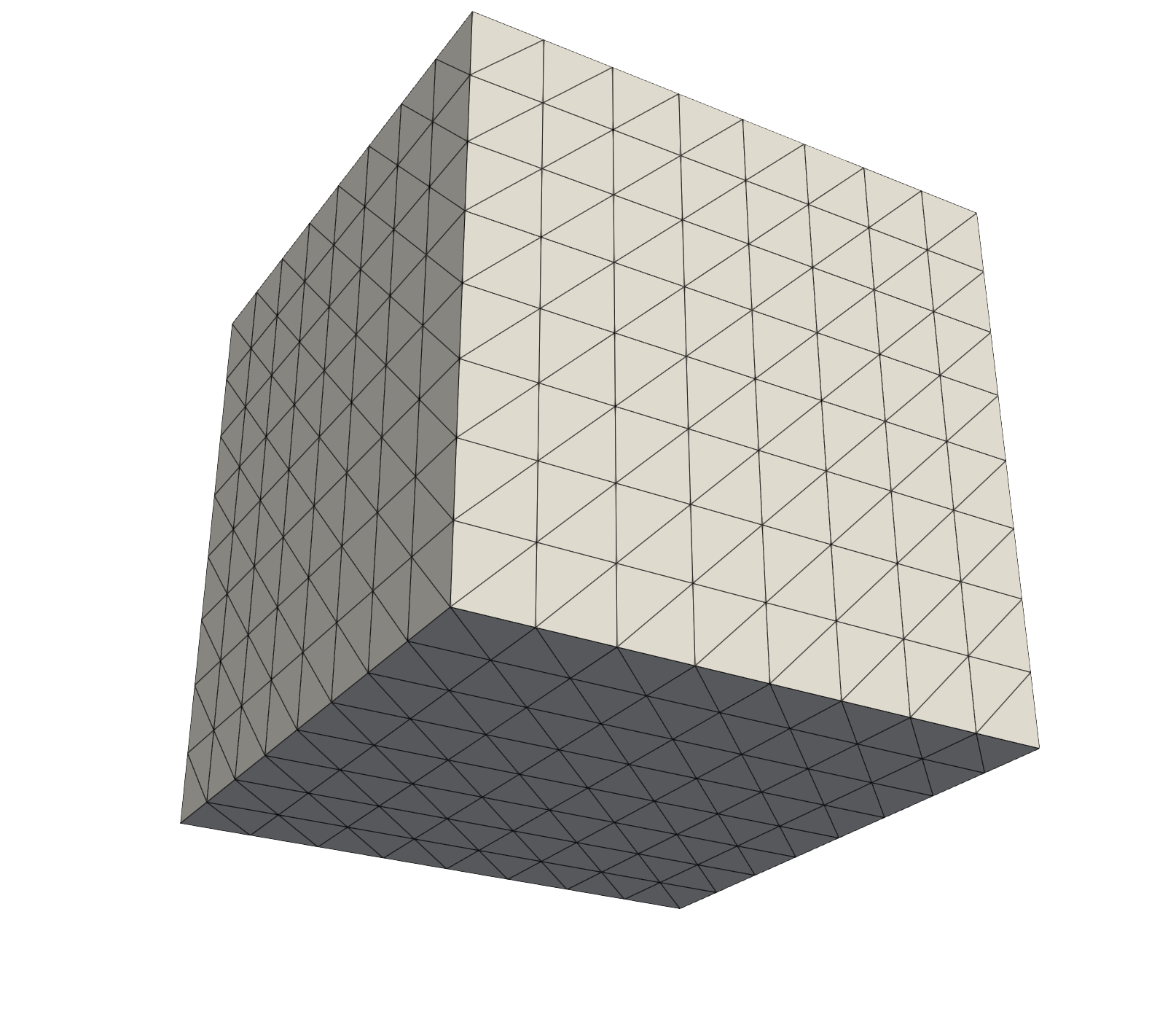}
			\centering\includegraphics[height=4.4cm, width=4.4cm]{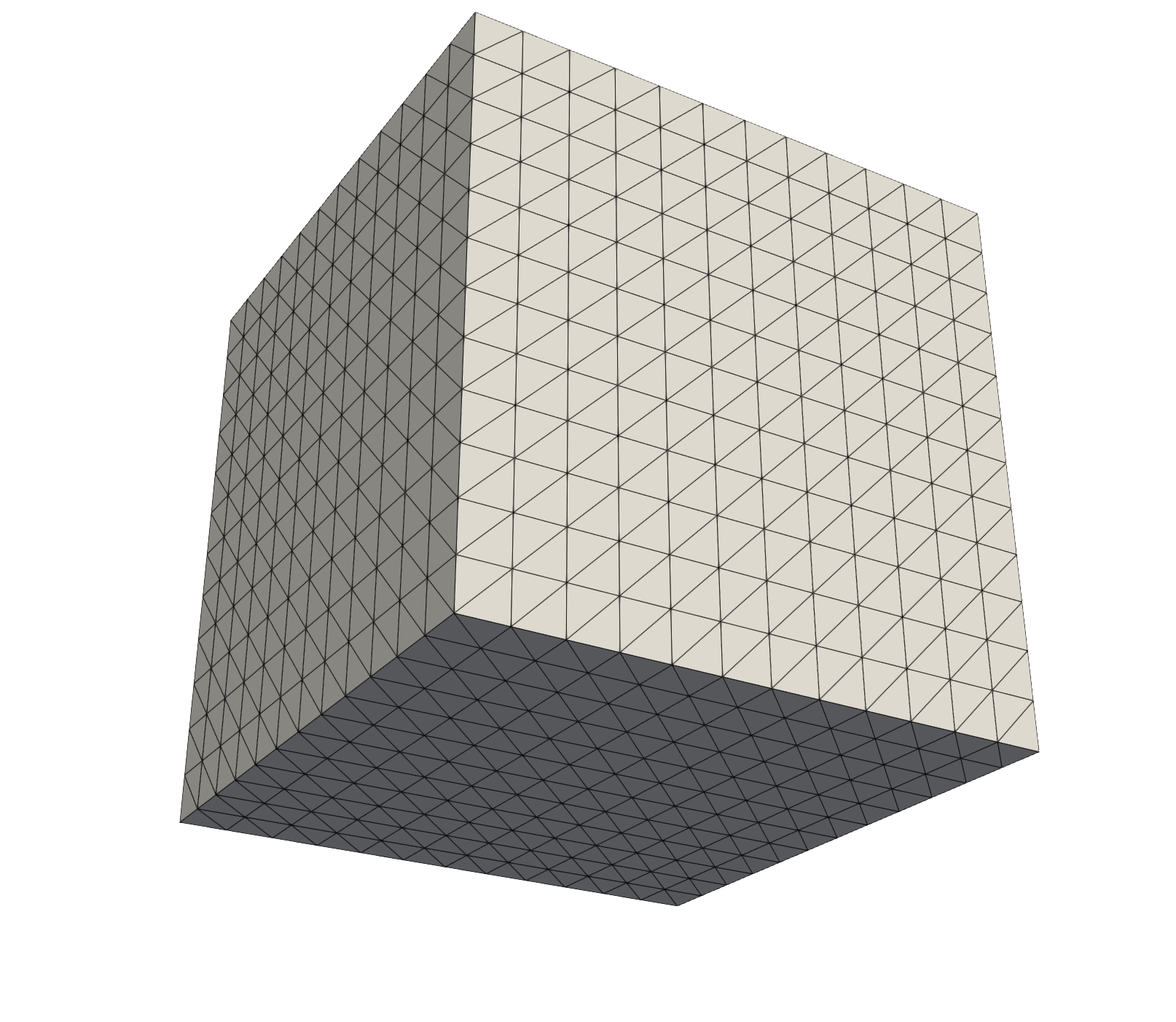}
                   \end{minipage}
		\caption{Examples of the meshes used in the unit cube. The left figure represents a mesh for $N=4$, the figure in for $N=8$  and the right figure for $N=12$.}
		\label{FIG:meshescube}
	\end{center}
\end{figure}

In Table \ref{table:cube} we report the first five  vibration frequencies  computed
for different values of $\nu$ and the corresponding orders of convergence and extrapolated values, considering the lowest order of approximation ($k=0$).
\begin{table}[H]
\footnotesize
\singlespacing
\begin{center}
\begin{tabular}{c |c c c c |c| c }
\toprule
 $\nu $        & $N=14$             &  $N=16$         &   $N=18$         & $N=20$ & $\alpha$& $\omega_{extr}$  \\ 
 \midrule
                                                 &4.43174&   4.43807&   4.44251 &  4.44573&   1.83 &  4.46093 \\
                                                 &4.44878&   4.45122&   4.45296 &  4.45424&   1.72 &  4.46068\\
 \multirow{2}{1cm}{0.35}           &4.44878&   4.45122&  4.45296  & 4.45424 &  1.72  & 4.46068\\
                                                 &4.76330&   4.76410&   4.76617 &  4.76702&   1.91 &  4.77083\\
                                                 &4.76645&   4.76740&   4.76807 &  4.76856&   1.83 &  4.77085\\
          
          \hline

                                            &4.61097  & 4.61289 & 4.61423 & 4.61521&   1.87&   4.61968\\
                                            &4.61396  & 4.61518 & 4.61604 & 4.61667&   1.79&   4.61970\\
 \multirow{2}{1cm}{0.45}      &4.61396  & 4.61518 & 4.61604 & 4.61667&   1.79&   4.61970\\
                                            &5.13647  & 5.15367 & 5.16565 & 5.17432&   1.88&   5.21395\\
                                            &5.17493  & 5.18332 & 5.18919 & 5.19345&   1.85&   5.21328  \\   
\hline

                                            &4.54302&  4.54513  & 4.54660 & 4.54767  &    1.85 &  4.55266\\
                                            &4.54594&  4.54735  & 4.54836 & 4.54909  &    1.76 &  4.55271\\
\multirow{2}{1cm}{0.5}         &4.54594 &  4.54735  &4.54836  & 4.54909  &1.76   & 4.55271\\
                                            &5.52484&  5.52524  & 5.52551 & 5.52571 &	 2.13   &5.52646\\
                                            &5.52484&  5.52524  & 5.52551 & 5.52571 &	 2.13   &5.52646\\
       
\bottomrule             
\end{tabular}
\end{center}
\caption{Computed lowest vibration frequencies for $k=0$ and different Poisson's ratio in the unitary cube.}
\label{table:cube}
\end{table}

It is clear from Table \ref{table:cube} that the double order of convergence for the eigenvalues is obtained in this geometry setting.
Also, for the limit case $\nu=0.5$, the method in the three dimensional domain works perfectly and approximates the eigenvalues 
with the expected double order $\mathcal{O}(h^2)$.

In Figure \ref{FIG:cubes} we present plots of the computed eigenfunctions in the unitary cube where the colors are as in the previous example. Also the plots show the deformation of the cube for each eigenfunction.
\begin{figure}[H]
	\begin{center}
		\begin{minipage}{11cm}
		\centering\includegraphics[height=4.4cm, width=5.4cm]{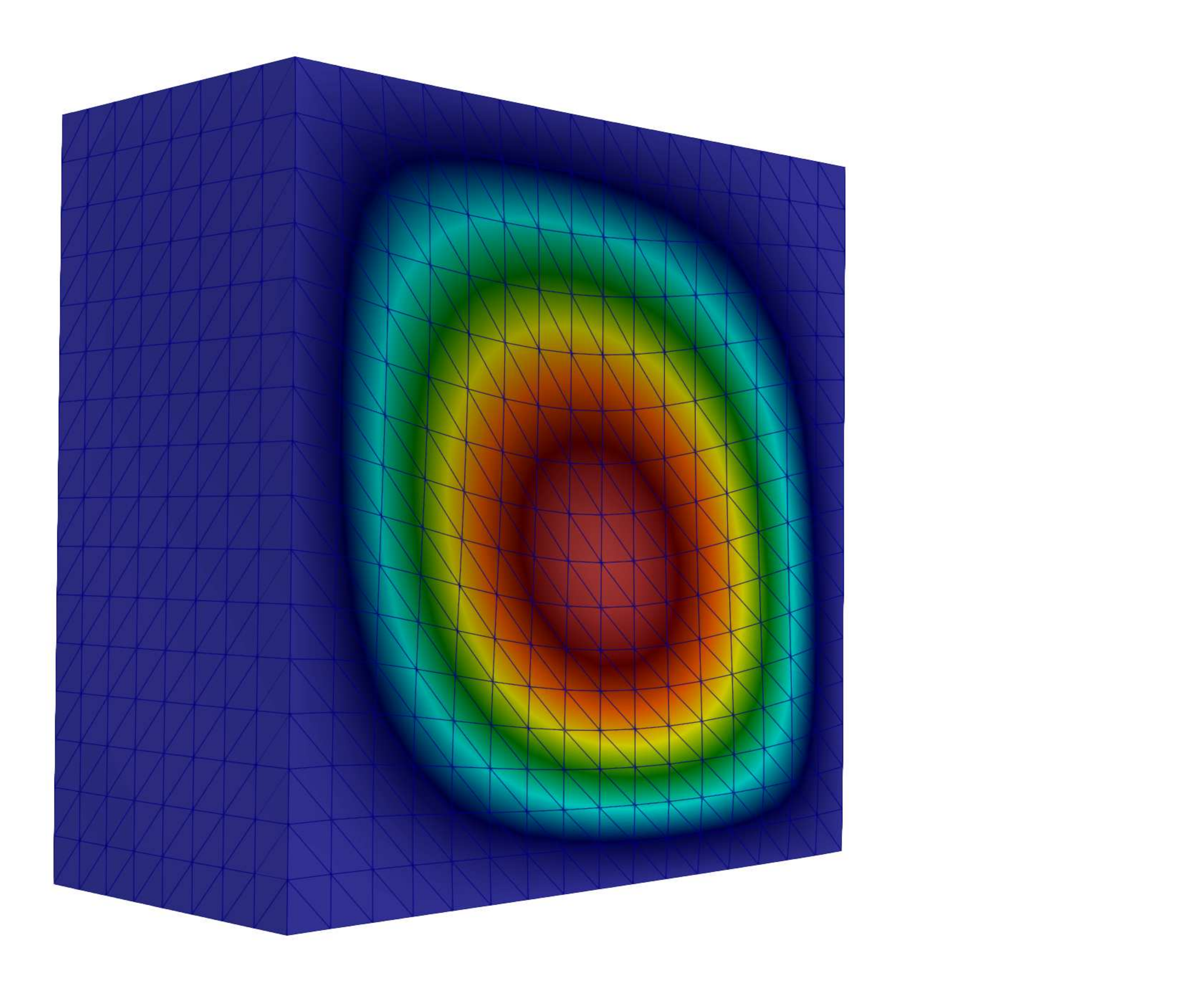}
		\centering\includegraphics[height=4.4cm, width=5.4cm]{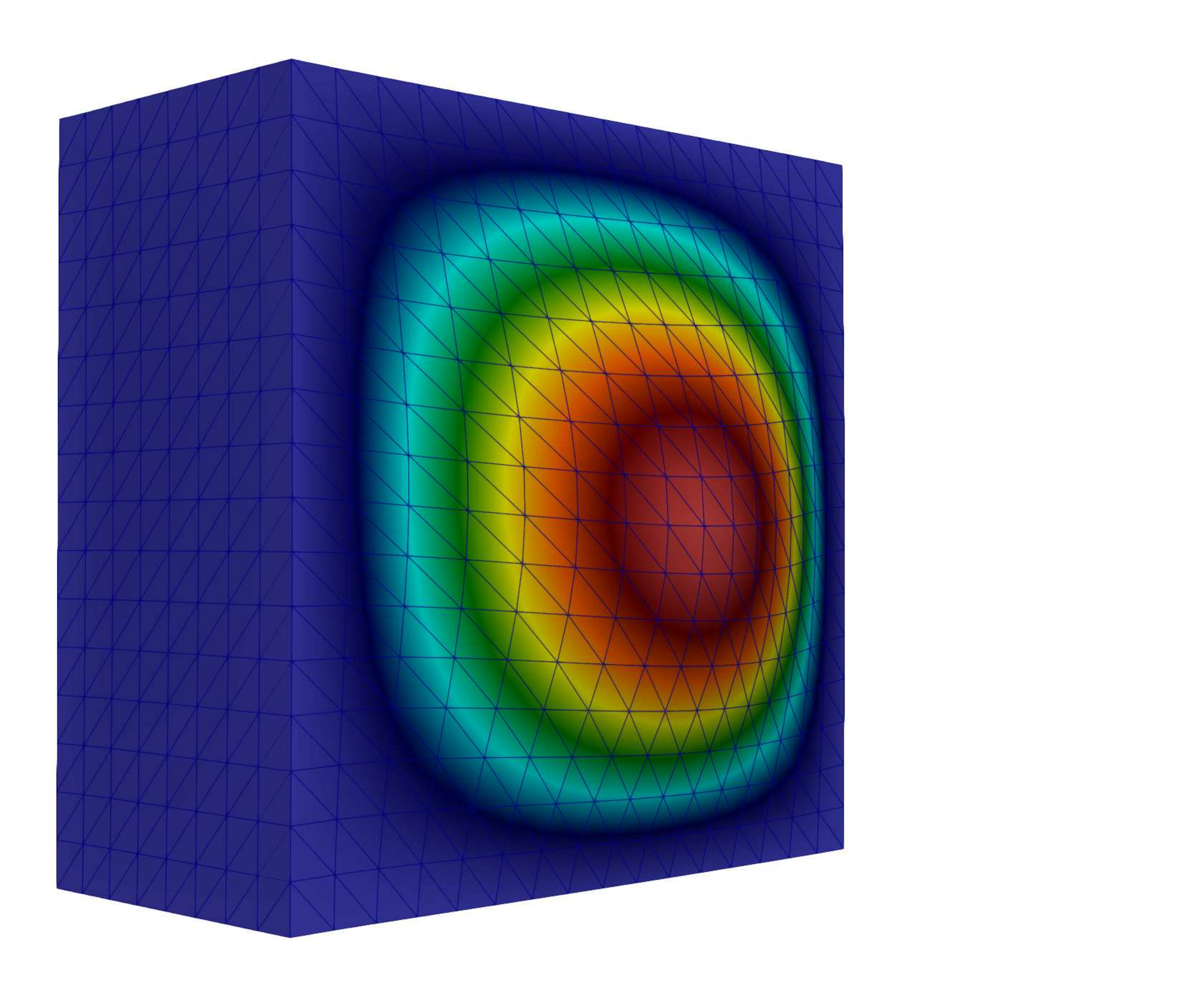}
			\centering\includegraphics[height=4.4cm, width=5.4cm]{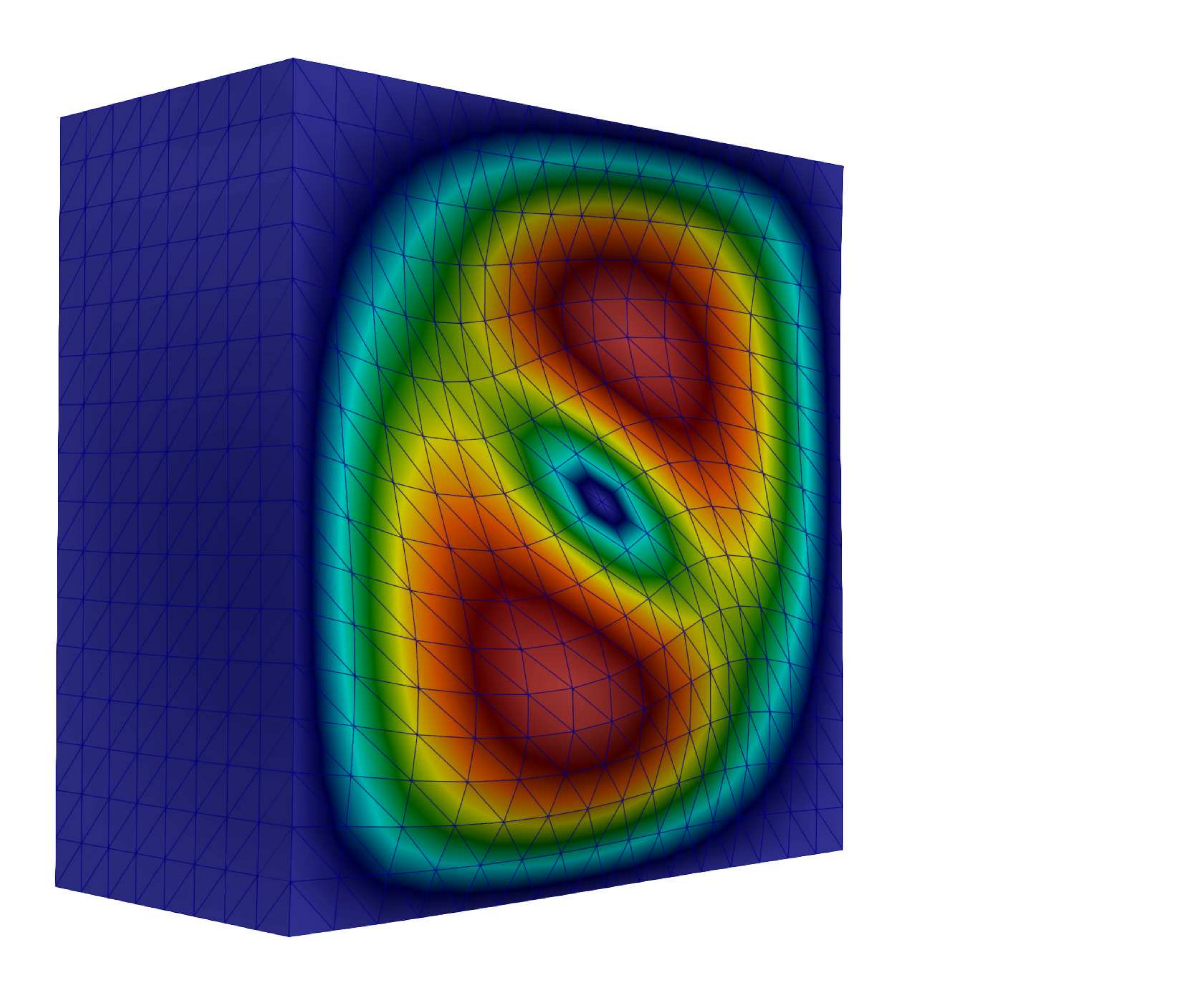}
			\centering\includegraphics[height=4.4cm, width=5.4cm]{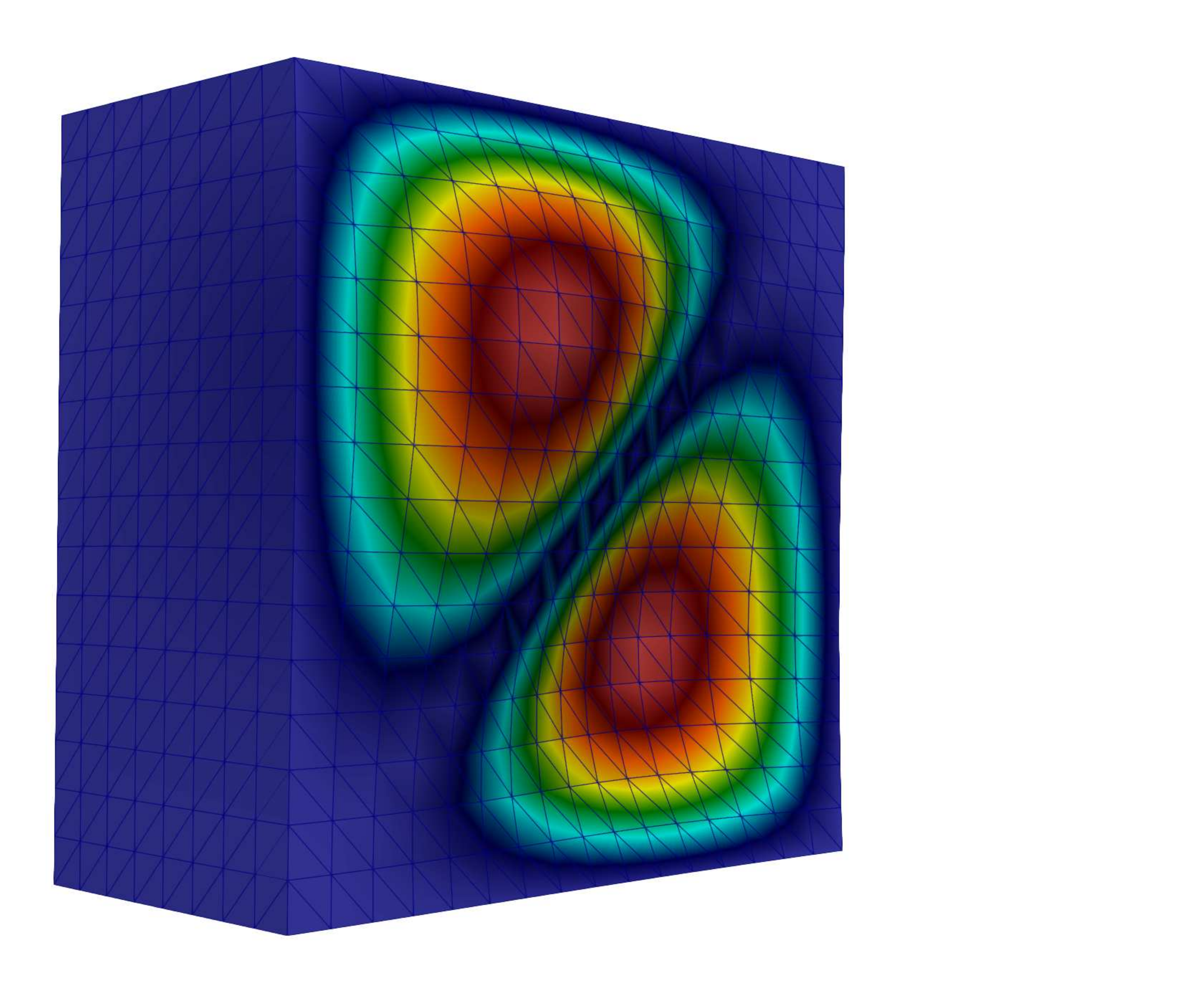}
                   \end{minipage}
		\caption{Eigenfunctions corresponding to the first (upper left), second and third (upper right), fourth (bottom left) and fifth (bottom right) eigenvalues, for $\nu=0.35$, $N=10$ and $k=0$.}
		\label{FIG:cubes}
	\end{center}
\end{figure}
\subsection{Circular domain}

For this test, we will consider as domain the unitary circle $\Omega_C:=\{(x,y)\in\mathbb{R}^2\,:\, x^2+y^2 <1\}$. Clearly, with this test we are considering a particular case where we will approximate a curved domain with triangles. This geometrical features will be reflected in the order of convergence, as it happens, for instance, in \cite{MR4077220} for the DG method.

In Tables \ref{table:circle1}, \ref{table:circle2} and \ref{table:circle3}, the parameter $N$ represents the refinement level of each mesh that, in this case, is such that $N$  is proportional to  $1/h$, where $h$ is the mesh size. In Figure \ref{meshes_circle} we present plots of some meshes for $\Omega_C$.

\begin{figure}[H]
	\begin{center}
		\begin{minipage}{14cm}
			\centering\includegraphics[height=3.4cm, width=3.4cm]{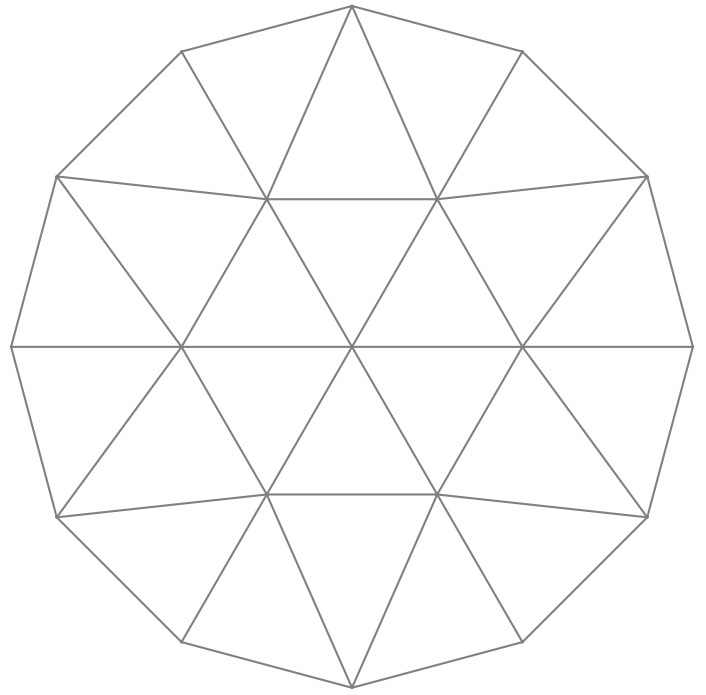}
			\centering\includegraphics[height=3.4cm, width=3.4cm]{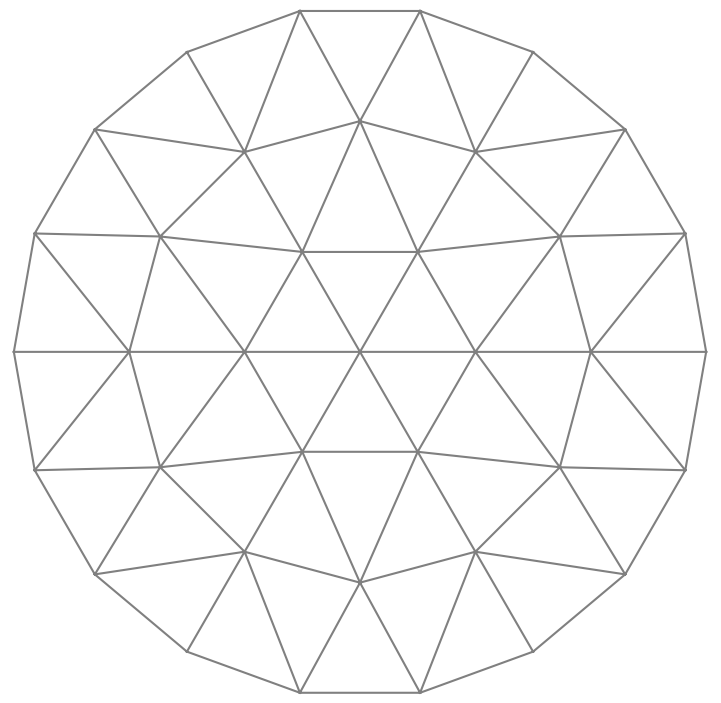}
			\centering\includegraphics[height=3.4cm, width=3.4cm]{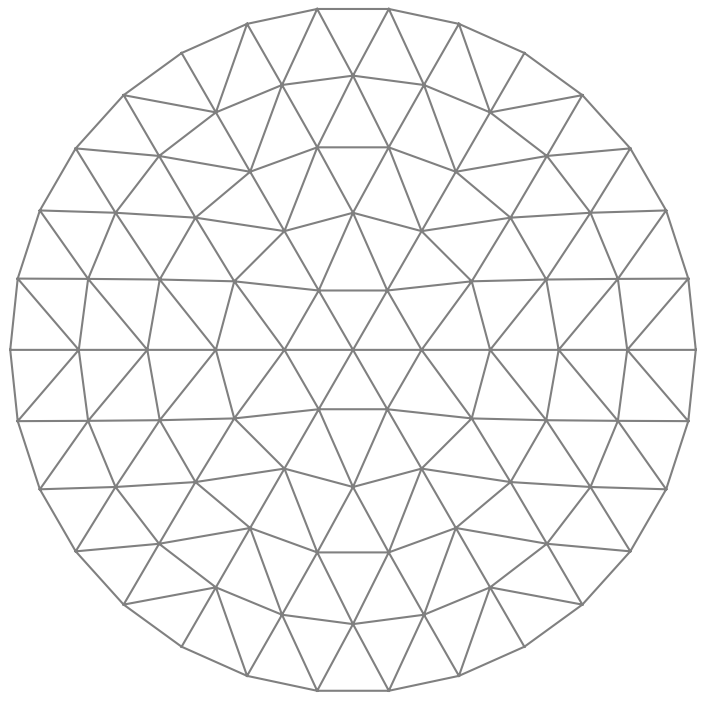}
                   \end{minipage}
		\caption{Examples of the meshes used in the circular domain.}
		\label{meshes_circle}
	\end{center}
\end{figure}

In the following table we report the computed first five vibration frequencies with our method, considering different values of $\nu$ and different polynomial degrees.
\begin{table}[H]
\footnotesize
\singlespacing
\begin{center}
\begin{tabular}{c |c c c c |c| c }
\toprule
 $\nu $        & $N=20$             &  $N=30$         &   $N=40$         & $N=50$ & $\alpha$& $\omega_{extr}$  \\ 
 \midrule
                                                 &2.33142 &  2.33193 &  2.33211  &  2.33216&   2.40 &  2.33225\\
                                                 &2.33142 &  2.33193 &  2.33211  &  2.33219&   1.98 &  2.33234\\
 \multirow{2}{1cm}{0.35}           &2.33344 &  2.33250 &  2.33229 &  2.33219&   2.25 &  2.33110\\
                                                 &3.32033 &  3.31881 &  3.31829  & 3.31805 &  2.02  & 3.31762\\
                                                 &3.32033 &  3.31881 &  3.31829  & 3.31805 &  2.02  & 3.31762\\
          
          \hline

                                            &2.22115 &  2.22032 &  2.22002 &   2.21989&   2.00&   2.21965\\
                                            &2.95910 &  2.95854 &  2.95834 &  2.95825 &  1.99 &  2.95809\\
 \multirow{2}{1cm}{0.49}      &2.95910&   2.95854&   2.95834&   2.95825&   1.99&   2.95809\\
                                            &3.68395 &  3.68339 &  3.68319 &  3.68310 &  1.89 &  3.68291\\
                                            &3.68395 &  3.68339 &  3.68319 &  3.68310 &  1.89 &  3.68291 \\
\hline

                                            &2.21374 &  2.21290 &  2.21261&   2.21248 &  2.00 &  2.21224\\
                                            &2.96637 &  2.96564 &  2.96538&   2.96526 &  2.00 &  2.96505\\
\multirow{2}{1cm}{0.5}         &2.96637 &  2.96564 &  2.96538&   2.96526 &  2.00 &  2.96505\\
                                            &3.68490 &  3.68419 &  3.68393&   3.68381 &  1.92 &  3.68358\\
                                            &3.68490 &  3.68419 &  3.68393&   3.68381 &  1.92 &  3.68358\\
       
\bottomrule             
\end{tabular}
\end{center}
\caption{Computed lowest vibration frequencies for $k=0$ and different Poisson's ratio in the unitary circle.}
\label{table:circle1}
\end{table}

\begin{table}[H]
\footnotesize
\singlespacing
\begin{center}
\begin{tabular}{c |c c c c |c| c }
\toprule
 $\nu $        & $N=10$             &  $N=20$         &   $N=30$         & $N=40$ & $\alpha$& $\omega_{extr}$  \\ 
 \midrule
                                                 &2.33244 &  2.33214&   2.33204&   2.33199&   2.00 &  2.33190\\
                                                 &2.33287 &  2.33258&   2.33247&   2.33243&   2.00 &  2.33234\\
 \multirow{2}{1cm}{0.35}           &2.33287 &  2.33258&   2.33247&   2.33243&   2.00 &  2.33234\\
                                                 &3.31838 &  3.31795&   3.31781&   3.31774&   2.01 &  3.31762\\
                                                 &3.31838 &  3.31795&   3.31781&   3.31774&   2.01 &  3.31762\\
          
          \hline

                                            &2.22016 &  2.21987 & 2.21977&  2.21973 &  2.00 &  2.21965\\
                                            &2.95877 &  2.95839 & 2.95826&  2.95820 &  2.01 &  2.95809\\
 \multirow{2}{1cm}{0.49}      &2.95877 &  2.95839 & 2.95826&  2.95820 &  2.01 &  2.95809\\
                                            &3.68378 &  3.68330 & 3.68314&  3.68306 &  2.02 &  3.68293\\
                                            &3.68378 &  3.68330 & 3.68314&  3.68306 &  2.02 &  3.68293  \\
\hline

                                            &2.21274 &  2.21246&  2.21236&  2.21232 &  2.00 &  2.21224\\
                                            &2.96573 &  2.96535&  2.96522&  2.96516 &  2.01 &  2.96505\\
\multirow{2}{1cm}{0.5}         &2.96573 &  2.96535&  2.96522&  2.96516 &  2.01 &  2.96505\\
                                            &3.68444 &  3.68396&  3.68380&  3.68372 &  2.02 &  3.68359\\
                                            &3.68444 &  3.68396&  3.68380&  3.68372 &  2.02 &  3.68359\\
       
\bottomrule             
\end{tabular}
\end{center}
\caption{Computed lowest vibration frequencies for $k=1$ and different Poisson's ratio in the unitary circle.}
\label{table:circle2}
\end{table}

\begin{table}[H]
\footnotesize
\singlespacing
\begin{center}
\begin{tabular}{c |c c c c |c| c }
\toprule
 $\nu $        & $N=10$             &  $N=20$         &   $N=30$         & $N=40$ & $\alpha$& $\omega_{extr}$  \\ 
 \midrule
                                                 &2.33243 &  2.33213 &  2.33203&   2.33198 &  2.01&   2.33190\\
                                                 &2.33287 &  2.33257 &  2.33247&   2.33242 &  2.00&   2.33234\\
 \multirow{2}{1cm}{0.35}           &2.33287 &  2.33257 &  2.33247&   2.33242 &  2.00&   2.33234\\
                                                 &3.31837 &  3.31795 &  3.31780&   3.31773 &  2.00&   3.31761\\
                                                 &3.31837 &  3.31795 &  3.31780&   3.31773 &  2.00&   3.31761\\
          
          \hline

                                            &2.22015 &  2.21987 &  2.21977 &  2.21972&   2.01&    2.21964\\
                                            &2.95876 &  2.95839&   2.95825 &  2.95819&    2.01&   2.95809\\
 \multirow{2}{1cm}{0.49}      &2.95876 &  2.95839 &  2.95825 &  2.95819&   2.01&   2.95809\\
                                            &3.68377 &  3.68330 &  3.68313 &  3.68305&   2.01&   3.68292\\
                                            &3.68377 &  3.68330 &  3.68313 &  3.68305&   2.01&   3.68292\\
\hline

                                            &2.21274&   2.21246&   2.21236 &  2.21231&  2.01 &  2.21223\\
                                            &2.96573&   2.96535&   2.96522 &  2.96516 &  2.01  & 2.96505\\
\multirow{2}{1cm}{0.5}        &2.96573&    2.96535&   2.96522  & 2.96516 &  2.01 & 2.96505\\
                                            &3.68443&   3.68396&   3.68379&  3.68372 &  2.01  & 3.68358\\
                                            &3.68443&   3.68396&   3.68379&  3.68372 &  2.01  & 3.68358\\
       
\bottomrule             
\end{tabular}
\end{center}
\caption{Computed lowest vibration frequencies for $k=2$ and different Poisson's ratio in the unitary circle.}
\label{table:circle3}
\end{table}

We observe from tables \ref{table:circle1},  \ref{table:circle2} and  \ref{table:circle3} that for different values of the Poisson ratio, even in the limit case when $\lambda=+\infty$, the
proposed method approximates with high accuracy the eigenvalues in the circle. An important phenomena in this experiment is that, independent of the polynomial
that we are considering, the order of convergence is $\mathcal{O}(h^2)$ for any $k\geq 0$ and Poisson ratio $\nu$. We remark that we obtain these orders  of convergence because of the variational crime committed by approximating the curved domain with a polygonal one. 

\begin{figure}[H]
	\begin{center}
		\begin{minipage}{14cm}
		
			\hspace{-0.9cm}\centering\includegraphics[height=4.4cm, width=4.1cm]{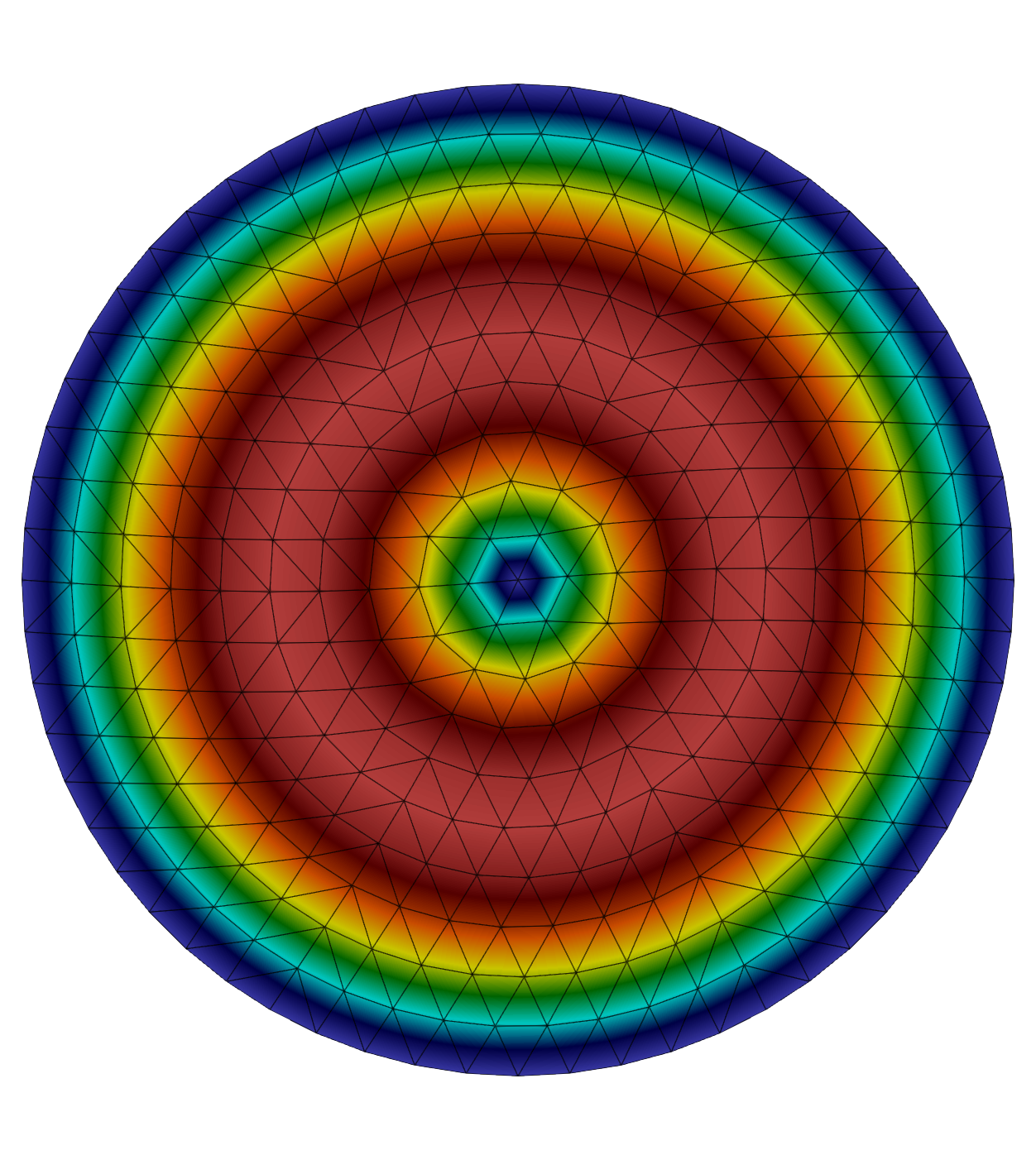}
			\centering\includegraphics[height=4.4cm, width=4.1cm]{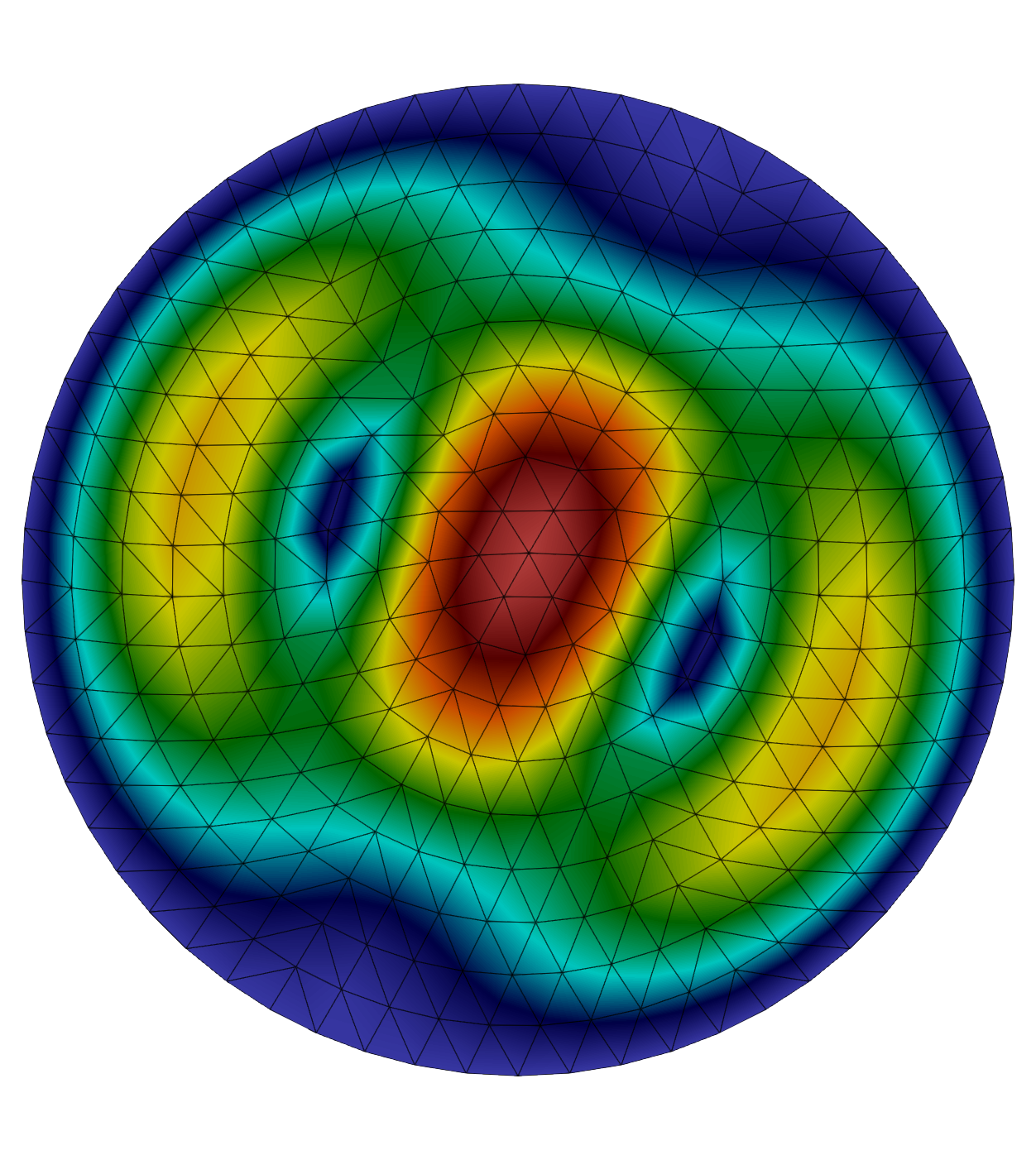}
			\centering\includegraphics[height=4.4cm, width=4.1cm]{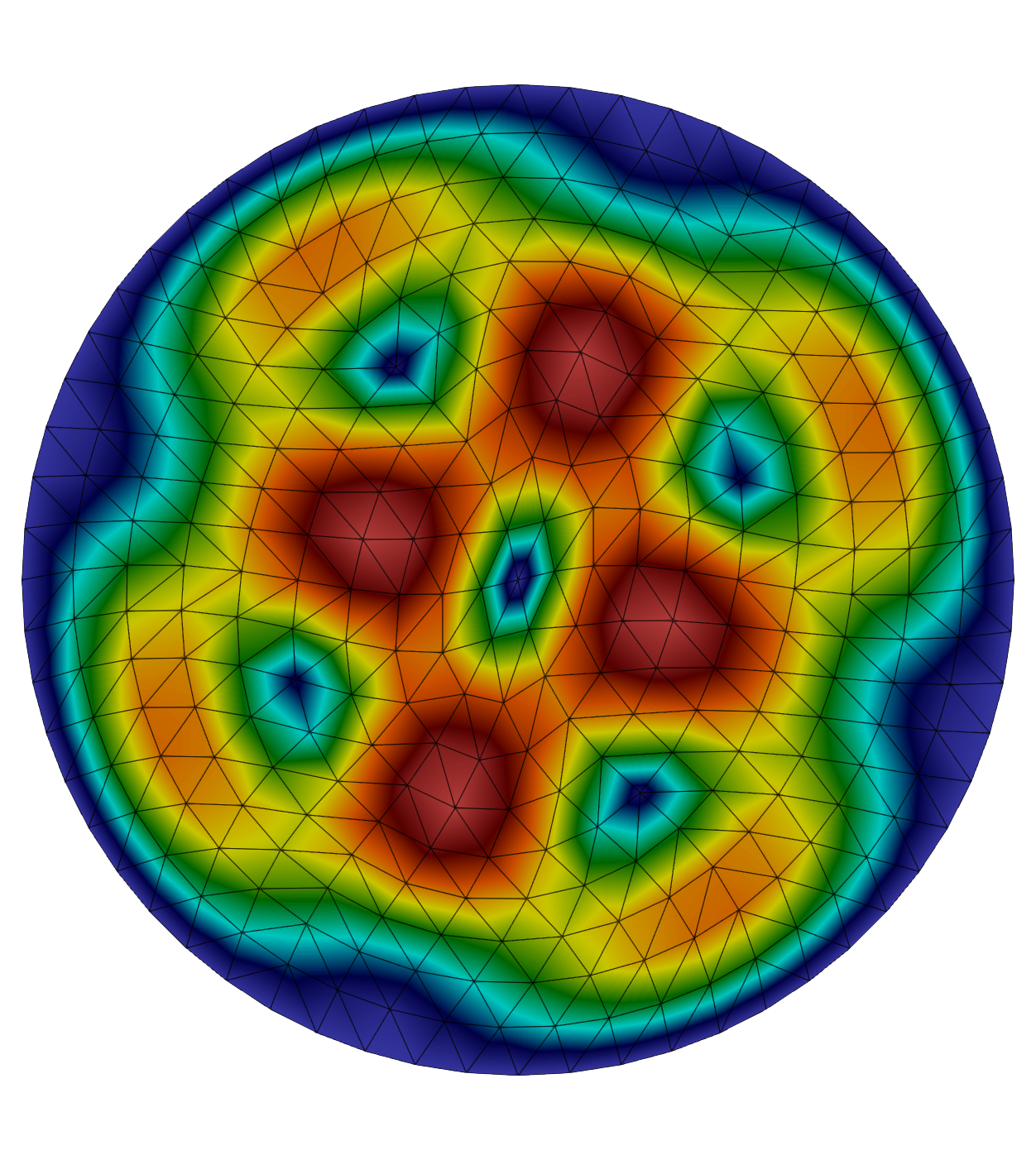}
                   \end{minipage}
		\caption{Eigenfunctions corresponding to the first (upper left), second and third (upper right), and fourth and fifth (bottom)  computed eigenfunctions with $\nu=0.49$, $N=10$ and $k=1$.}
		\label{FIG:plotcircle}
	\end{center}
\end{figure}

\subsection{L-shaped domain}  In this section we consider a non-convex domain that we will call the L-shaped domain which is defined by $\Omega_L:= (-1,1)^2\setminus [-1,0]^2.$

\begin{table}[H]
\footnotesize
\begin{center}
\begin{tabular}{c c}
\toprule
 $\nu$ & $\ws$  \\ 
 \midrule
0.35 & 0.6797\\
0.49&0.5999\\
0.5&0.5946\\
\bottomrule             
\end{tabular}
\end{center}
\caption{Sobolev regularity exponents.}
\label{TABLA:s0}
\end{table}

The eigenfunctions of this problem may present singularities  due the reentrant angles of the domain. According to \cite{MR840970} in this case the estimate Lemma \ref{lmm:add_eigen} holds true for all $s<0.5445$. Comparing this value with those of Table \ref{TABLA:s0}, it is observed that the strongest singularity can arise from the reentrant angle.  Then, the theoretical order of convergence satisfies $2s\geq 1.08$.

In this experiment,  $N$ will represent the refinement level of the meshes  and it is chosen as similar to the circular domain. In tables \ref{table:k=0_cuadrado_L1}, \ref{table:k=1_cuadrado_L2} and \ref{table:k=2_cuadrado_L3} we report the first five vibration frequencies  obtained with our method, the respective order of convergence and extrapolates values for $\nu=0.35, 0.49, 0.5$ and polynomial
degrees $k=0,1,2$.

\begin{table}[H]
\begin{center}
\begin{tabular}{c |c c c c |c| c }
\toprule
 $\nu $        & $N=10$             &  $N=20$         &   $N=30$         & $N=40$ & $\alpha$& $\omega_{extr}$  \\ 
 \midrule
					 & 2.35882&   2.37007  & 2.37385&   2.37452 & 1.35&   2.37768\\
                                           & 2.76390&   2.78752  & 2.79281&   2.79434 & 1.79 &  2.79726\\
 \multirow{2}{1cm}{0.35}     &3.19541  &  3.24358  & 3.26016&   3.26354 &  1.28 & 3.27876\\
  					 &3.56499  &  3.60428  & 3.61377 &  3.61600  & 1.74 & 3.62146\\
                                           &3.73545  &  3.77218  & 3.78016 &  3.78270  & 1.80 & 3.78710\\
          \hline

  					   &3.17244&   3.22333&   3.24295&   3.24667&  1.15&  3.26734\\
  					   &3.43339&   3.48637&   3.49814&   3.50156&  1.80&  3.50800\\
 \multirow{2}{1cm}{0.49}      &3.67087&    3.70410&   3.71116 &  3.71350 & 1.82 & 3.71731\\
					   &3.97744&   4.02607&   4.03584&   4.03817&  2.00&  4.04256\\
					   &4.11516&   4.17884&   4.19661&   4.20125&  1.51&  4.21421\\

\hline

					&3.17263 &  3.22526&   3.24565&   3.24948  & 1.14 &  3.27131\\
					&3.44052 &  3.49132&   3.50264&   3.50594  & 1.79 &  3.51223\\
\multirow{2}{1cm}{0.5}       &3.70134&   3.72848&   3.73388&   3.73579 &  1.89&   3.73851\\
					&3.97643 &  4.02421&   4.03397&   4.03625  & 1.98 &  4.04072\\
                                          &4.22529 &  4.26701&   4.28257&   4.28629  & 1.13 &  4.30359\\
\bottomrule             
\end{tabular}
\end{center}
\caption{Computed lowest vibration frequencies for $k=0$ and different Poisson's ratio in the L-shaped domain.}
\label{table:k=0_cuadrado_L1}
\end{table}

\begin{table}[H]
\begin{center}
\begin{tabular}{c |c c c c |c| c }
\toprule
 $\nu $        & $N=10$             &  $N=20$         &   $N=30$         & $N=40$ & $\alpha$& $\omega_{extr}$  \\ 
 \midrule
					 &2.37477 &  2.37699&   2.37774 &  2.37779&  1.49&   2.37830\\
                                           &2.79548 &  2.79706&   2.79739 &  2.79746&  1.96&   2.79761\\
 \multirow{2}{1cm}{0.35}     &3.26609 &  3.27349&   3.27588 &  3.27605&  1.53&   3.27766\\
  					 &3.61883 &  3.62068 &  3.62114  & 3.62123 & 1.75&   3.62149\\
                                           &3.78550 &  3.78619&   3.78641 &  3.78642&  1.60&   3.78655\\
          \hline

  					   &3.24882&   3.25954 &  3.26319 &  3.26346 &  1.46&   3.26609\\
  					   &3.50462&   3.50739 &  3.50798 &  3.50814 &  1.87&   3.50845\\
 \multirow{2}{1cm}{0.49}       &3.71621&   3.71680 &  3.71700 &  3.71701 &  1.51&   3.71715\\
					   &4.04169&   4.04245 &  4.04259 &  4.04261 &  2.14&   4.04267\\
					   &4.20692&   4.21080 &  4.21184 &  4.21194 & 1.76 &    4.21251\\

\hline

					&3.25169 &  3.26296&   3.26682&   3.26710&   1.45 &  3.26991\\
					&3.50883 &  3.51157&   3.51216&   3.51232&   1.88 &  3.51261\\
\multirow{2}{1cm}{0.5}       &3.73771 &  3.73839&   3.73856&   3.73857&   1.29 &  3.73877\\
					&4.03966 &  4.04051&   4.04067&   4.04070&   2.13 &  4.04076\\
 				        &4.29036. &  4.29473&   4.29598&  4.29607&   1.70 &  4.29679\\
\bottomrule             
\end{tabular}
\end{center}
\caption{Computed lowest vibration frequencies for $k=1$ and different Poisson's ratio in the L-shaped domain.}
\label{table:k=1_cuadrado_L2}
\end{table}

\begin{table}[H]
\begin{center}
\begin{tabular}{c |c c c c |c| c }
\toprule
 $\nu $        & $N=10$             &  $N=20$         &   $N=30$         & $N=40$ & $\alpha$& $\omega_{extr}$  \\ 
 \midrule
					 & 2.37704 & 2.37798&   2.37831&   2.37833 &  1.44&  2.37857\\
                                           & 2.79708 & 2.79752&   2.79761&   2.79763 &  1.90&  2.79768\\
 \multirow{2}{1cm}{0.35}     &3.27366  & 3.27668 &  3.27773 &  3.27778  & 1.46&  3.278538\\
  					 &3.62068  & 3.62132 &  3.62148 &  3.62150  & 1.82&  3.62158\\ 
                                           &3.78619  & 3.78647 &  3.78657 &  3.78658  & 1.45&  3.78665\\
          \hline

  					   &3.25982&   3.26450 &  3.26619 &  3.26628&   1.41&   3.26754\\
  					   &3.50741&   3.50826 &  3.50843 &  3.50847&   2.05&   3.50854\\
 \multirow{2}{1cm}{0.49}       &3.71679&   3.71707 &  3.71717 &  3.71718&   1.35&   3.71726\\
					   &4.04246&   4.04263 &  4.04266 &  4.04266&   2.60&   4.04267\\
					   &4.21091&   4.21223 &  4.21269 &  4.21272&   1.43&   4.21306\\

\hline

					&3.26326 &  3.26821&   3.27000&   3.27009 &  1.40&   3.27146\\
					&3.51160 &  3.51243&   3.51259 &   3.51263  & 2.08&  3.51270\\
\multirow{2}{1cm}{0.5}      & 3.73830 &  3.73864&   3.73877&   3.73877 &  1.35&   3.73888\\
					&4.04053 &  4.04072&   4.04075&   4.04076 &  2.57&   4.04076\\
                                          &4.29486 &  4.29642&   4.29699&   4.29702 &  1.38&   4.29746\\
\bottomrule             
\end{tabular}

\end{center}
\caption{Computed lowest vibration frequencies for $k=2$ and different Poisson's ratio in the L-shaped domain.}
\label{table:k=2_cuadrado_L3}
\end{table}

We observe from Tables \ref{table:k=0_cuadrado_L1}, \ref{table:k=1_cuadrado_L2} and \ref{table:k=2_cuadrado_L3} that our method provides a double order of convergence for the vibration frequencies. Namely, in all cases we have $s \thickapprox 2 \min\{r, k+1\}$, which corresponds to the the best possible order of convergence for this problem. 

We end this section presenting plots of the first three eigenfunctions obtained with our method in the L-shaped domain.
In particular,  we show the eigenfunctions computed with $\nu=0.35$, $k=1$ as polynomial degree and $N=10$.

\begin{figure}[H]
	\begin{center}
		\begin{minipage}{14cm}
			\hspace{-1.4 cm}\centering\includegraphics[height=4.0cm, width=4.0cm]{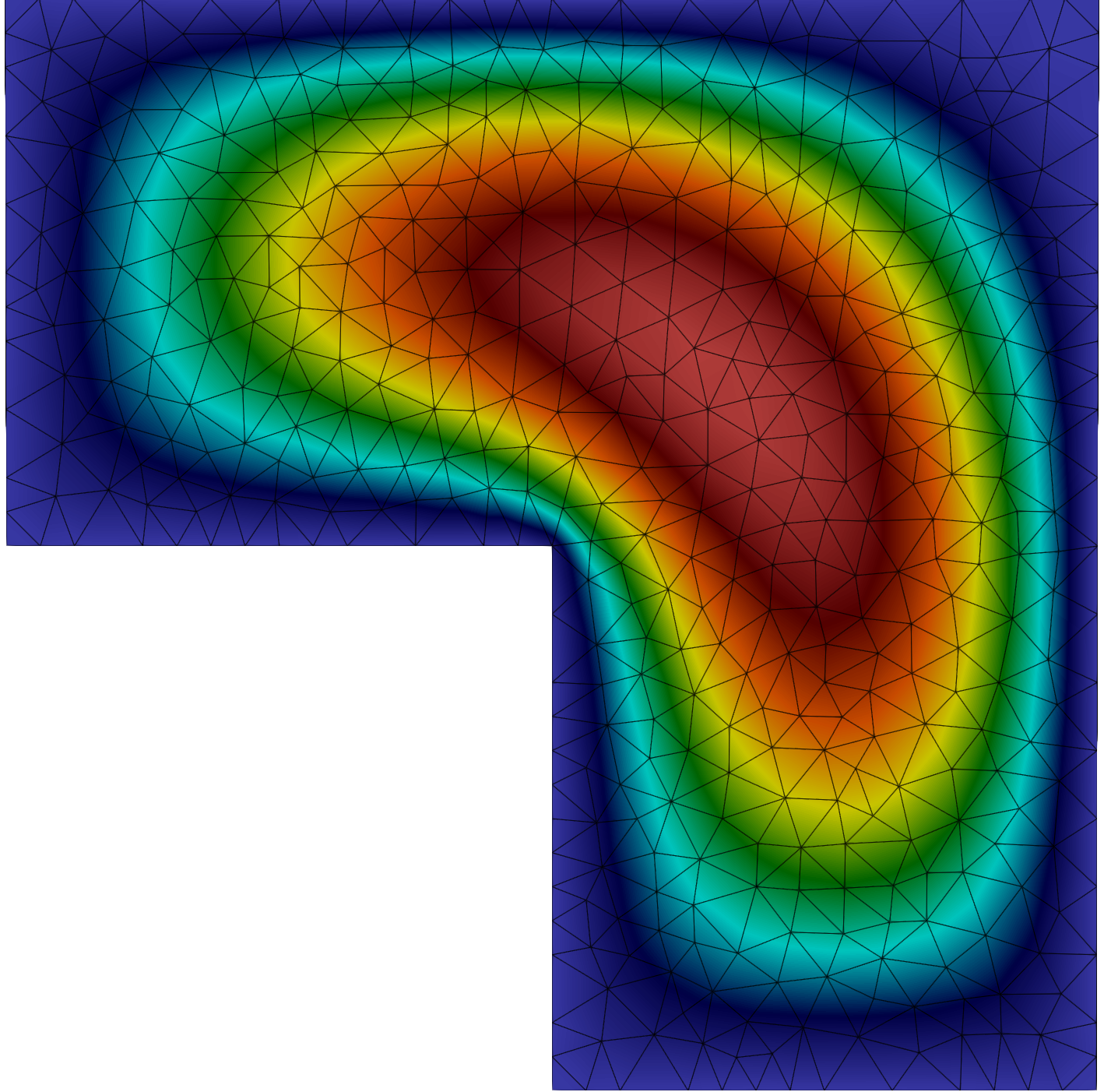}\hspace{0.3 cm}
			\centering\includegraphics[height=4.0cm, width=4.0cm]{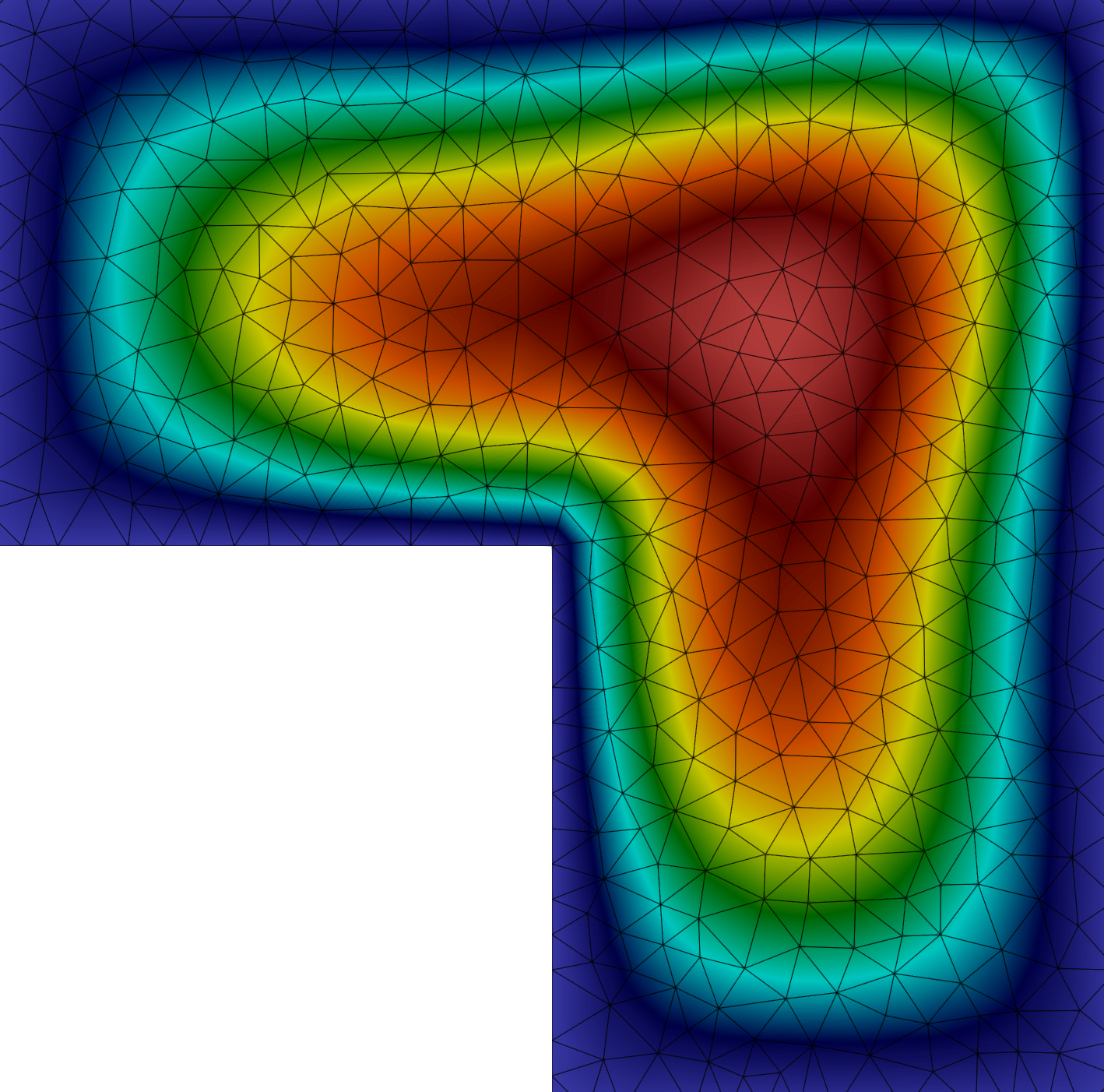}\hspace{0.3 cm}
			\centering\includegraphics[height=4.0cm, width=4.0cm]{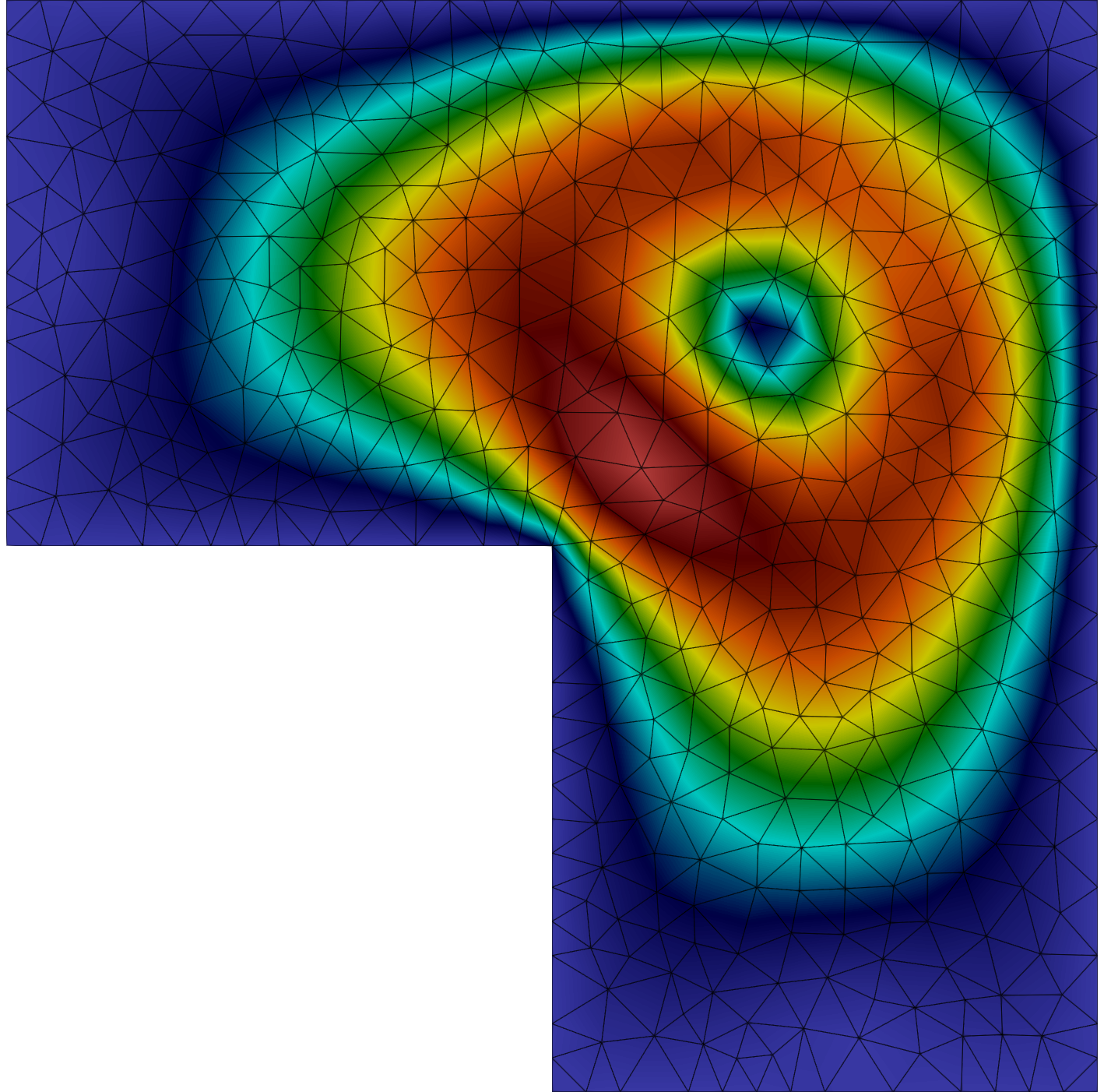}
                   \end{minipage}
		\caption{Eigenfunctions corresponding to the first (upper left), second (middle) and third (right) computed eigenfunctions with $\nu=0.35$, $N=10$ and $k=1$.}
		\label{FIG:plotL}
	\end{center}
\end{figure}


\bibliographystyle{siam}
\footnotesize
\bibliography{ILR_revised}

\end{document}